\documentclass[12pt]{amsart}
\usepackage{amsmath,amssymb,amsfonts,amsthm,amsopn}
\usepackage{graphicx,tikz}
\usepackage[all]{xy}
\usepackage{amsmath}
\usepackage{multirow, longtable, makecell, caption, array}
\usepackage{amssymb}
\usepackage{mathtools}

\usepackage{cellspace} %
\def\Spnr{\mathrm{Sp}(2d,\R)}
\def\Gltwonr{\mathrm{GL}(2d,\R)}

\newcommand{\tfa}{time-frequency analysis}

\newcommand{\ft}{Fourier transform}
\newcommand{\stft}{short-time Fourier transform}
\newcommand{\psdo}{pseudodifferential operator}

\newcommand{\tpsdo}{$\tau$-pseudo\-differential operator}

\newcommand{\tf}{time-frequency}
\newcommand{\opt }{\mathrm{Op}_{\tau}}

\newcommand{\tfs}{time-frequency shift}

\newtheorem{tm}{Theorem}[section]
\newtheorem{lemma}[tm]{Lemma}
\newtheorem{prop}[tm]{Proposition}

\newtheorem{theorem}{Theorem}[section]
\newtheorem{corollary}[theorem]{Corollary}
\newtheorem{definition}[theorem]{Definition}
\newtheorem{example}[theorem]{Example}
\newtheorem{proposition}[theorem]{Proposition}
\newtheorem{remark}[theorem]{Remark}
\newtheorem*{thm*}{Theorem}
\newtheorem*{cor*}{Corollary}
\newcommand{\beqa}{\begin{eqnarray*}}
	\newcommand{\eeqa}{\end{eqnarray*}}

\newcommand{\field}[1]{\mathbb{#1}}
\newcommand{\bR}{\field{R}}
\newcommand{\bN}{\field{N}}
\newcommand{\bZ}{\field{Z}}

\def\G{\mathcal{G}}

\def\la{\lambda}

\def\cF{\mathcal{F}}
\def\cS{\mathcal{S}}

\def\cB{\mathcal{B}}

\def\cU{\mathcal{U}}
\def\cA{\mathcal{A}}

\def\cA{\mathcal{A}}

\def\rd{\bR^d}

\def\rdd{{\bR^{2d}}}
\def\zdd{{\bZ^{2d}}}

\def\intrdd{\int_{\rdd}}

\def\R{\right)}

\def\<{\left<}
\def\>{\right>}

\def\inv{^{-1}}

\def\mv1{M_v^1}

\def\Lmpq{L_m^{p,q}}

\def\Mmpq{M_m^{p,q}}
\def\phas{(x,\omega )}

\def\o{\omega}

\def\R{\mathbb{R}}
\def\Ren{\mathbb{R}^d}
\def\Renn{\mathbb{R}^{2d}}

\def\sch{\mathcal{S}}

\def\Fur{\mathcal{F}}

\def\f{\varphi}

\def\Sn2{S_{2}(L^{2}(\Ren))}
\def\S1{S_{1}(L^{2}(\Ren))}
\def\sig00{\sigma_{0,0}}

\def\la{\langle}
\def\ra{\rangle}

\newcommand{\A}{\mathcal{A}}

\begin{document}
	\begin{abstract} In this paper we focus on the almost-diagonalization properties of $\tau$-pseudodifferential operators using techniques from \tfa.
		Our function spaces are modulation spaces and the special class of Wiener amalgam spaces arising by considering the action of the Fourier transform of modulation spaces.
		Such spaces are nowadays called modulation spaces as well.
		A particular example is provided by the Sj\"ostrand class, for which Gr\"ochenig \cite{Grochenig_2006_Time} exhibits the almost diagonalization of Weyl operators. We shall show that such result can be extended to any $\tau$-pseudodifferential operator, for $\tau \in [0,1]$. This is not surprising, since the mapping that goes from a Weyl symbol to a $\tau$-symbol is bounded in the Sj\"ostrand class. What is  new and quite  striking is the almost diagonalization for $\tau$-operators with symbols in weighted Wiener amalgam spaces. In this case 
		the diagonalization depends on the parameter $\tau$. In particular, we have an almost diagonalization for $\tau\in (0,1)$ whereas the cases $\tau=0$ or $\tau=1$ yield only to weaker results. As a consequence, we infer boundedness, algebra and Wiener properties for $\tau$-\psdo s on Wiener amalgam and modulation spaces. 
	\end{abstract}
	
	\title[Almost diagonalization of $\tau$-pseudodifferential operators]{Almost diagonalization of $\tau$-pseudodifferential operators with symbols in Wiener amalgam and modulation spaces}
	\author{Elena Cordero}
	\address{Dipartimento di Matematica, Universit\`a di Torino, Dipartimento di
		Matematica, via Carlo Alberto 10, 10123 Torino, Italy}
	\email{elena.cordero@unito.it}
	\thanks{}
	\author{Fabio Nicola}
	\address{Dipartimento di Scienze Matematiche, Politecnico di Torino, corso
		Duca degli Abruzzi 24, 10129 Torino, Italy}
	\email{fabio.nicola@polito.it}
	\thanks{}
	\author{S. Ivan Trapasso}
	\address{Dipartimento di Scienze Matematiche, Politecnico di Torino, corso
		Duca degli Abruzzi 24, 10129 Torino, Italy}
	\email{salvatore.trapasso@polito.it}
	\thanks{}
	
	\date{}
\maketitle

\section{Introduction }
 It is widely known that the techniques of Time-frequency Analysis revealed to be very fruitful when dealing with diverse problems in quite different fields, from signal processing to harmonic analysis, but also from PDE's to quantum mechanics (see, e.g., \cite{BogetalTRANS,CNStricharzJDE2008,deGossonDiasPrata2014,deGossonWigner2017,Ruzhansky2016,SugimotoWang2011,SugimotoWang2011}). Besides the manifold achievements, Gabor analysis is a fascinating discipline in itself: the study of related functional-analytic problems is an inexhausted source of challenges and it happens to cast new light on established results, eventually. As an example, Gr\"{o}chenig employed these methods in his work \cite{Grochenig_2006_Time}, almost diagonalizing  Weyl operators with symbols in the  Sj\"{o}strand class via Gabor frames. This result  has provided new insights on a number of well-known outcomes  previously obtained by Sj\"{o}strand within the realm of classical analysis \cite{Sjo94,Sjo95} and, more importantly,  has lead to far-reaching generalizations. 

In this work we extend part of the results obtained in \cite{Grochenig_2006_Time} and derive a number of relevant consequences on the boundedness of pseudodifferential operators. In particular, our study has been carried out within the unifying framework offered by $\tau$-representations and $\tau$-pseudodifferential operators. $\tau$-pseudodifferential  operators can be either defined as a quantization rule or by means of the related time-frequency representation (cf.\ \cite{BogetalTRANS,BoggiattoPseudo2010,BoggiattoCubo2010}). In the following section we give a more detailed account, but here we simply recall the latter.
For $\tau \in \lbrack 0,1]$, the (cross-)$\tau $-Wigner distribution is given by
\begin{equation}
W_{\tau }(f,g)(x,\omega )=\int_{\mathbb{R}^{d}}e^{-2\pi iy\omega }f(x+\tau y)%
\overline{g(x-(1-\tau )y)}\,dy,\quad f,g\in \mathcal{S}(\mathbb{R}^{d}),
\label{tauwig}
\end{equation}
whereas the \tpsdo\ is  defined by
\begin{equation}
\langle \opt  (a)f,g\rangle =
\langle
a,W_{\tau }(g,f)\rangle, \quad f,g\in \mathcal{S}(\mathbb{R}^{d}).
\label{tauweak}
\end{equation}
For $\tau=1/2$  we recapture the Weyl operator studied in \cite{Grochenig_2006_Time}, if  $\tau=0$ the operator is called the Kohn-Nirenberg operator $\mathrm{Op_{KN}}$, if $\tau=1$ the related operator is named operator \textquotedblleft with right symbol\textquotedblright.

For what concerns the symbol classes, we consider Banach spaces allowing
a suitable measure of the time-frequency decay of distributions, i.e.
modulation and Wiener amalgam spaces. These function spaces were introduced by Feichtinger
in the '80s (cf. \cite{Segal81.Feichtinger_1981_Banach,feichtinger1983modulation}) and revealed to be the optimal setting for a large class
of problems related to harmonic analysis and PDE's (cf. \cite{deGossonsymplectic2011,deGossonGRomero2016,Wangbook2011}). 

To fix notation, we write a point in phase space (in \tf\ space) as
$z=(x,\omega)\in\rdd$, and  the corresponding phase-space shift (\tfs )
acting on a function or distribution  as
\begin{equation}
\label{eq:kh25}
\pi (z)f(t) = e^{2\pi i \omega t} f(t-x), \, \quad t\in\rd.
\end{equation}
When considering symbol functions, we shall work with \tfs s $\pi(z,\zeta)$, with $z,\zeta\in\rdd$ (the variables are doubled).

Consider a Schwartz  function $g\in\cS(\rdd)\setminus \{0\}$. We define the modulation
space $M^{\infty,1}(\rdd)$ (or Sj\"{o}strand class) as the space of tempered distributions $\sigma\in\cS'(\rdd)$ such that
$$ \intrdd \sup_{z\in\rdd} |\la \sigma, \pi(z,\zeta)g\ra| d \zeta<\infty.
$$
For the properties of such space and its weighted versions we refer to the next section.
Here the main  focus is on symbols in the so-called Wiener amalgam
space $W(\cF L^\infty, L^1)(\rdd) =\cF M^{\infty,1}(\rdd)$, where $\cF$ is the Fourier transform. Heuristically, symbols in $W(\cF L^\infty, L^1)(\rdd)$ display locally a regularity of the type $\cF L^\infty(\rdd)$ and globally decay as functions in $L^1(\rdd)$. For instance, the $\delta$ distribution  (in $\cS'(\rdd)$) belongs to $W(\cF L^\infty, L^1)(\rdd)$.  Nowadays this class can be referred
to as modulation space as well (although in a generalized sense, as suggested
by Feichtinger in his retrospective \cite{Feich2006}), see also \cite{Pfander2013}.

The central issue of our work is the approximate diagonalization of
$\tau$-pseudo\-differential operators. This is a well-known problem, studied in several contexts of harmonic analysis: classical references are  \cite{meyer1990ondelettes,rochberg1998pseudodifferential}. See also, for the more general framework of Fourier integral operators, \cite{Labate2008,cordero2013wiener} and references therein. In somewhat heuristic terms, the choice of certain type of symbols assures that these operators preserve the time-frequency localization, since their kernel with
respect to continuous or discrete time-frequency shifts satisfies a convenient decay condition. To be precise, thanks to a covariance property for the $\tau$-Wigner distribution under the action of time-frequency
shifts, we are able to prove the following claim, which extends  \cite[Thereom 3.2]{Grochenig_2006_Time} proved for  $\tau=\tfrac{1}{2}$, i.e., for the Weyl quantization (see the subsequent Theorem \ref{teor41} for the weighted version).

First, we recall the definition of a Gabor frame. Let $\Lambda=A\zdd$,  with $A\in \mathrm{GL}(2d,\R)$, be a lattice of the time-frequency plane.
  The set  of
 time-frequency shifts $\G(\f,\Lambda)=\{\pi(\lambda)\f:\
 \lambda\in\Lambda\}$ for a  non-zero $\f\in L^2(\rd)$ (the so-called window function) is named a Gabor system. The set $\G(\f,\Lambda)$   is
 a Gabor frame, if there exist
 constants $A,B>0$ such that
 \begin{equation}\label{gaborframe}
 A\|f\|_2^2\leq\sum_{\lambda\in\Lambda}|\langle f,\pi(\lambda)\f\rangle|^2\leq B\|f\|^2_2,\qquad \forall f\in L^2(\rd).
 \end{equation}
 Our result is the following:
\begin{thm*}
	Fix a non-zero window $\varphi\in \cS(\rd)$ such that 
	 $\mathcal{G}\left(\varphi,\Lambda\right)$ is a Gabor frame  for $L^{2}\left(\mathbb{R}^{d}\right)$.
	For any $\tau\in\left[0,1\right]$, the following properties are equivalent:
	\begin{enumerate}
		\item[$(i)$] $\sigma\in M^{\infty,1}\left(\mathbb{R}^{2d}\right)$.
		\item[$(ii)$] $\sigma\in\mathcal{S}'\left(\mathbb{R}^{2d}\right)$ and there exists
		a function $H_{\tau}\in L^{1}\left(\mathbb{R}^{2d}\right)$ such
		that
		\[
		\left|\left\langle \mathrm{Op}_{\tau}\left(\sigma\right)\pi\left(z\right)\varphi,\pi\left(w\right)\varphi\right\rangle \right|\le H_{\tau}\left(w-z\right),\qquad\forall w,z\in\mathbb{R}^{2d}.
		\]
		\item[$(iii)$] $\sigma\in\mathcal{S}'\left(\mathbb{R}^{2d}\right)$ and there exists
		a sequence $h_{\tau}\in\ell^{1}\left(\Lambda\right)$ such that
		\[
		\left|\left\langle \mathrm{Op}_{\tau}\left(\sigma\right)\pi\left(\mu\right)\varphi,\pi\left(\lambda\right)\varphi\right\rangle \right|\le h_{\tau}\left(\lambda-\mu\right),\qquad\forall\lambda,\mu\in\Lambda.
		\]
	\end{enumerate}
\end{thm*}

This kind of control is a characterizing property of symbols in those
spaces, in the following sense. 
\begin{cor*}
	Under the hypotheses of the previous Theorem, let $\tau \in [0,1]$ be fixed and assume that \, $T:\mathcal{S}\left(\mathbb{R}^{d}\right)\rightarrow\mathcal{S}'\left(\mathbb{R}^{d}\right)$
	is continuous and satisfies one of the following conditions:
	\begin{itemize}
		\item[$(i)$] $\left|\left\langle T\pi\left(z\right)\varphi,\pi\left(w\right)\varphi\right\rangle \right|\le H\left(w-z\right),\quad\forall w,z\in\mathbb{R}^{2d}$
		for some $H\in L^{1}$. 
		\item[$(ii)$] $\left|\left\langle T\pi\left(\mu\right)\varphi,\pi\left(\lambda\right)\varphi\right\rangle \right|\le h\left(\lambda-\mu\right),\quad\forall\lambda,\mu\in\Lambda$
		for some $h\in\ell^{1}$. 
	\end{itemize}
	Therefore, $T=\mathrm{Op}_{\tau}\left(\sigma\right)$ for some symbol
	$\sigma\in M^{\infty,1}\left(\mathbb{R}^{2d}\right)$. 
\end{cor*}

We remark that, if either the previous condition $(i)$ or $(ii)$ is satisfied, then it is well known that the operator $T$ has Weyl symbol in the {S}j\"ostrand's class (\cite[Thereom 3.2]{Grochenig_2006_Time}). On the other hand, the mapping that goes to a Weyl symbol to a $\tau$-symbol is bounded in the {S}j\"ostrand's class, and the previous corollary follows from \cite[Thereom 3.2]{Grochenig_2006_Time}. Novelty in this connection mainly resides in the weighted version, cf. Theorem \ref{teor41} and Corollary \ref{corteor41} below.

A parallel pattern can be traced for symbols in $W(\cF L^\infty, L^1)(\rdd)$ although it is somewhat more involved because of the peculiar
way $\tau$ comes across.  This is not surprising if we notice that Weyl operators with symbols 
in $L^1(\rdd)\subset W(\cF L^\infty, L^1)(\rdd)$ are bounded on $L^2(\rd)$ (cf. \cite{Wangbook2011}) but neither Kohn-Nirenberg operators ($\tau=0$) nor operators  \textquotedblleft with  right symbol\textquotedblright ($\tau=1$) are, cf. \cite{Boulkhemairpseudo1995, BoulkhemairCPDE1997,BoulkhemairMRL1997}, see also \cite{EIFL2}.

For $\tau\in (0,1)$, we  introduce the symplectic matrix 
\begin{equation}\label{UT}
\mathcal{U}_{\tau}=-\left(\begin{array}{cc}
\frac{\tau}{1-\tau}I_{d\times d} & 0_{d\times d}\\
0_{d\times d} & \frac{1-\tau}{\tau}I_{d\times d}
\end{array}\right)\in\mathrm{Sp}\left(2d,\mathbb{R}\right)
\end{equation}
(see the next section for details). Our result is as follows:
\begin{thm*}
	Fix  non-zero window $\varphi\in\mathcal{S}\left(\mathbb{R}^{d}\right)$.
	For any $\tau\in\left(0,1\right)$, the following properties are equivalent:
	\begin{enumerate}
		\item[$(i)$] $\sigma\in W\left(\mathcal{F}L^{\infty},L^{1}\right)\left(\mathbb{R}^{2d}\right)$.
		\item[$(ii)$] $\sigma\in\mathcal{S}'\left(\mathbb{R}^{2d}\right)$ and there exists
		a function $H_{\tau}\in L^{1}\left(\mathbb{R}^{2d}\right)$ such
		that
	\begin{equation}\label{almostdieq}
		\left|\left\langle \mathrm{Op}_{\tau}\left(\sigma\right)\pi\left(z\right)\varphi,\pi\left(w\right)\varphi\right\rangle \right|\le H_{\tau}\left(w-\mathcal{U}_{\tau}z\right),\qquad\forall w,z\in\mathbb{R}^{2d}.
	\end{equation}
	\end{enumerate}
\end{thm*}

As easy example, we shall consider $\sigma=\delta\in W\left(\mathcal{F}L^{\infty},L^{1}\right)$  (cf. Example \ref{esempio} below). In this case,  it is easy to see that $\mathrm{Op}_{\tau}(\delta)$ is a linear change of variables and  the estimate in \eqref{almostdieq} is an equality.

Another difference from the Sj\"ostrand symbol class concerns
the discrete almost diagonalization, recovered only if $\tau=\tfrac{1}{2}$, see the subsequent Corollary \ref{discreteWeyl}. 

Even though  $$w-\mathcal{U}_{\tau}z=0$$ is not exactly the diagonal, by abuse of notation we use the term \emph{almost diagonalization} also in this framework.  In fact, \eqref{almostdieq} can be suitably interpreted as an estimate on the concentration of the time-frequency representation $\opt (\sigma)$ along the graph of $\cU_{\tau}$.  

This kind of control on the kernels paves the way for proving relevant
properties of $\tau$-operators. Here we mainly focus on boundedness results in different settings, where Wiener amalgam and modulation spaces play the role of symbol class and domain/target. In particular, we underline the connection with the theory of Fourier integral operators (FIOs): we have that, under suitable conditions, $\tau$-pseudodifferential operators admit a representation as type I FIOs, thus inheriting a number of features
(including boundedness, algebra and Wiener properties) established in previous papers by two of the authors - see for instance \cite{cnr,cordero2013wiener,generalizedmetaplectic}. In addition, we point out that even when this overlapping does not hold, the insightful proof strategies developed there can still be exploited. Apart from the technical assumptions on weights which will be specified later, our results can be collected in Table \ref{tabella}. 
\begin{center}
	\begin{table}[!h]
		\centering
		 \setcellgapes{3pt} \makegapedcells
	\begin{tabular}{|c|c|c|}
		\hline 
		$\boldsymbol{\tau}$ & \textbf{Symbol - $\left(\mathbb{R}^{2d}\right)$} & \textbf{Boundedness - $\left(\mathbb{R}^{d}\right)$}, \,$1\leq p,q\leq\infty$\tabularnewline
		\hline 
		\hline 
		$\left(0,1\right)$ & $\sigma\in W\left(\mathcal{F}L^{\infty},L^{1}\right)$ & $\mathrm{Op}_{\tau}\left(\sigma\right):M^{p,q}\rightarrow M^{p,q}$\tabularnewline
		\hline 
		\makecell{$\left[0,1\right]$} & \makecell{$\sigma\in M^{\infty,1}$} & \makecell{$\mathrm{Op}_{\tau}\left(\sigma\right):W\left(\mathcal{F}L^{p},L^{q}\right)\rightarrow W\left(\mathcal{F}L^{p},L^{q}\right)$}\tabularnewline
		\hline 
		\makecell{$\left(0,1\right)$} & \makecell{$\sigma\in W(\mathcal{F}L^{\infty},L^{1})$} & \makecell{ $\mathrm{Op}_{\tau}\left(\sigma\right):W\left(\mathcal{F}L^{p},L^{q}\right)$ 
	$\rightarrow W\left(\mathcal{F}L^{p},L^{q}\right)$}\tabularnewline
		\hline 
		0 & $\sigma\in W\left(\mathcal{F}L^{\infty},L^{1}\right)$ & $\mathrm{Op}_{\mathrm{KN}}\left(\sigma\right):M^{1,\infty}\rightarrow M^{1,\infty}$\tabularnewline
		\hline 
		\multirow{1}{*}{1} & \multirow{1}{*}{$\sigma\in W\left(\mathcal{F}L^{\infty},L^{1}\right)$} & $\mathrm{Op}_{1}\left(\sigma\right):W\left(\mathcal{F}L^{1},L^{\infty}\right)\rightarrow W\left(\mathcal{F}L^{1},L^{\infty}\right)$\tabularnewline
		\hline 
	\end{tabular}
	    \bigskip
	\caption{Boundedness results proved in the paper. \label{tabella}}
\end{table}
	\par\end{center}

It is important to emphasize again the singular behaviour showed
by the end-points $\tau=0$ and $\tau=1$, which is essentially
due to weaker versions of the previous theorems, see  Propositions \ref{Prop5.8} and \ref{Prop5.9} in the sequel. 

We finally remark that, as said earlier, the results on approximate
diagonalization allow a potentially large number of directions to
be investigated, even in a more abstract framework (cf. for instance \cite{grochenig2008banach}). We wish
to explore some of these routes in subsequent papers.

\section{Preliminaries}
 \textbf{Notation.} We define $t^2=t\cdot t$, for $t\in\rd$, and
  $xy=x\cdot y$ is the scalar product on $\Ren$. The Schwartz class is denoted by  $\sch(\Ren)$, the space of tempered
  distributions by  $\sch'(\Ren)$.   We use the brackets  $\la
  f,g\ra$ to denote the extension to $\sch' (\Ren)\times\sch (\Ren)$ of
  the inner product $\la f,g\ra=\int f(t){\overline {g(t)}}dt$ on
  $L^2(\Ren)$. 
  The Fourier transform of a function $f$ on $\rd$ is normalized as
  \[
  \Fur f(\xi)= \int_{\rd} e^{-2\pi i x\xi} f(x)\, dx.
  \]

 The modulation $M_{\omega}$ and translation $T_{x}$ operators   are defined as 
 \[
 M_{\omega}f\left(t\right)= e^{2\pi it \omega}f\left(t\right),\qquad T_{x}f\left(t\right)= f\left(t-x\right).
 \]
 The following result exhibits composition and commutation properties of the \tfs s.
 \begin{lemma}
 	For $x,x',\omega,\omega'\in\mathbb{R}^{d}$, we have\\
 	(i) Commutation relations for TF-shifts:
 	\begin{equation}
 	T_{x}M_{\omega}=e^{-2\pi ix\omega}M_{\omega}T_{x}.\label{eq: fundid}
 	\end{equation}
 	(ii) Compositions of TF-shifts:
 	\begin{equation}
 	\pi\left(x,\omega\right)\pi\left(x',\omega'\right)=e^{-2\pi ix\omega'}\pi\left(x+x',\omega+\omega'\right).\label{eq: composizione pi}
 	\end{equation}
 	More explicitly,
 	\[
 	\left(M_{\omega}T_{x}\right)\left(M_{\omega'}T_{x'}\right)=e^{-2\pi ix\omega'}M_{\omega+\omega'}T_{x+x'}.
 	\]
 \end{lemma}

 Let $f\in\cS'(\rd)$. We define the short-time Fourier transform of $f$ as
 \begin{equation}\label{STFTdef}
 V_gf\phas=\langle f,\pi\phas g\rangle=\Fur (fT_x g)(\omega)=\int_{\Ren}
 f(y)\, {\overline {g(y-x)}} \, e^{-2\pi iy \o }\,dy.
 \end{equation}
 Recall the fundamental property of time-frequency analysis
 \begin{equation}\label{FI}
 V_{g}f\left(x,\omega\right)=e^{-2\pi ix\omega}V_{\hat{g}}\hat{f}\left(\omega,-x\right).
 \end{equation}

Denote by $J$ the canonical symplectic matrix in $\mathbb{R}^{2d}$:
\[
J=\left(\begin{array}{cc}
0_{d\times d} & I_{d\times d}\\
-I_{d\times d} & 0_{d\times d}
\end{array}\right)\in\mathrm{Sp}\left(2d,\mathbb{R}\right),
\]
where  the
symplectic group $\mathrm{Sp}\left(2d,\mathbb{R}\right)$ is defined
by
$$
\Spnr=\left\{M\in\Gltwonr:\;M^{\top}JM=J\right\}.
$$
Observe that, for $z=\left(z_{1},z_{2}\right)\in\mathbb{R}^{2d}$, we have
$Jz=J\left(z_{1},z_{2}\right)=\left(z_{2},-z_{1}\right),$  $J^{-1}z=J^{-1}\left(z_{1},z_{2}\right)=\left(-z_{2},z_{1}\right)=-Jz,$ and 
$J^{2}=-I_{2d\times2d}.$

 \begin{lemma}
 	~\label{lem:properties of Atau}For any $\tau\in (0,1)$,
 	consider the matrix 
 	\begin{equation}
 	\mathcal{A}_{\tau}=\left(\begin{array}{cc}
 	0_{d\times d} & \sqrt{\frac{1-\tau}{\tau}}I_{d\times d}\\
 	-\sqrt{\frac{\tau}{1-\tau}}I_{d\times d} & 0_{d\times d}
 	\end{array}\right).
 	\end{equation}
 	The following properties hold:
 	\begin{enumerate}
 		\item $\mathcal{A}_{\tau}\in\mathrm{Sp}\left(d,\mathbb{R}\right)$; in particular,
 		$\mathcal{A}_{1/2}=J$.
 		\item $\mathcal{A}_{\tau}^{\top}=-\mathcal{A}_{1-\tau}$, $\mathcal{A}_{\tau}^{-1}=-\mathcal{A}_{\tau}$.
 		\item $\mathcal{A}_{1-\tau}\mathcal{A}_{\tau}=\mathcal{A}_{\tau}^{\top}\mathcal{A}_{\tau}^{-1}=I_{2d\times2d}-\mathcal{B}_{\tau}$,
 		where 
 		\begin{equation}
 		\mathcal{B}_{\tau}=\left(\begin{array}{cc}
 		\frac{1}{1-\tau}I_{d\times d} & 0_{d\times d}\\
 		0_{d\times d} & \frac{1}{\tau}I_{d\times d}
 		\end{array}\right).\label{eq:Btau}
 		\end{equation}
 		\item $\sqrt{\tau\left(1-\tau\right)}\left(\mathcal{A}_{\tau}+\mathcal{A}_{1-\tau}\right)=\sqrt{\tau\left(1-\tau\right)}\mathcal{B}_{\tau}\mathcal{A}_{\tau}=J$.
 	\end{enumerate}
 	\end{lemma}
 	\begin{proof} These properties follow by easy computations. 
 		\end{proof}
 	Although $\mathcal{A}_{\tau}$ is well defined only for $\tau\in\left(0,1\right)$,
 	notice that the matrix 
 	\begin{equation}\label{A'}
 	\mathcal{A}_{\tau}'=\left(\begin{array}{cc}
 	0_{d\times d} & \left(1-\tau\right)I_{d\times d}\\
 	-\tau I_{d\times d} & 0_{d\times d}
 	\end{array}\right)
 	\end{equation}
 	is well defined for any $\tau\in\left[0,1\right]$, though non-invertible
 	and non-symplectic when $\tau\in\left\{ 0,1\right\} $. In particular:
 	\begin{enumerate}
 		\item For $\tau\in\left(0,1\right)$, $\mathcal{A}_{\tau}'=\sqrt{\tau\left(1-\tau\right)}\mathcal{A}_{\tau}$. 
 		\item $\left(\mathcal{A}_{\tau}'\right)^{\top}=-\mathcal{A}_{1-\tau}'$. 
 		\item $\mathcal{A}_{\tau}'+\mathcal{A}_{1-\tau}'=J$. 
 	\end{enumerate}
In the following we shall use the \stft\, of a $\tau$-Wigner distribution, which is contained in \cite{EIFL2}.
 \begin{lemma}
 	\label{lem:STFT of tauWig}Consider $\tau\in\left[0,1\right]$, $\varphi_{1},\varphi_{2}\in\mathcal{S}\left(\mathbb{R}^{d}\right)$,
 	$f,g\in\mathcal{S}\left(\mathbb{R}^{d}\right)$ and set $\Phi_{\tau}=W_{\tau}\left(\varphi_{1},\varphi_{2}\right)\in\mathcal{S}\left(\mathbb{R}^{2d}\right)$.
 	Then, 
 	\[
 	{V}_{\Phi_{\tau}}W_{\tau}\left(g,f\right)\left(z,\zeta\right)=e^{-2\pi iz_{2}\zeta_{2}}V_{\varphi_{1}}g\left(z-\mathcal{A}_{1-\tau}'\zeta\right)\overline{V_{\varphi_{2}}f\left(z+\mathcal{A}_{\tau}'\zeta\right),}
 	\]
 	where $z=\left(z_{1},z_{2}\right),\,\zeta=\left(\zeta_{1},\zeta_{2}\right)\in\mathbb{R}^{2d}$ and the matrix $\mathcal{A}_{\tau}'$ is defined in \eqref{A'}.
 	In particular, 
  	\[
 	\left|{V}_{\Phi_{\tau}}W_{\tau}\left(g,f\right)\right|=\left|V_{\varphi_{1}}g\left(z-\mathcal{A}_{1-\tau}'\zeta\right)\right|\cdot\left|V_{\varphi_{2}}f\left(z+\mathcal{A}_{\tau}'\zeta\right)\right|.
 	\]
 	
 \end{lemma}

\subsection{Function Spaces} We first need to introduce the weight functions under our consideration.\\ 
{\textbf {Weight functions.}}
Following \cite[Sec.\ 2.3]{Grochenig_2006_Time}, we will work with so-called \emph{admissible
weight functions}, i.e. by $v$ we always denote a non-negative continuous
function on $\mathbb{R}^{2d}$ such that 
\begin{enumerate}
\item $v\left(0\right)=1$ and $v$ is even in each coordinate: 
\[
v\left(\pm z_{1},\ldots,\pm z_{2d}\right)=v\left(z_{1},\ldots,z_{2d}\right).
\]
\item $v$ is submultiplicative, that is 
\[
v\left(w+z\right)\le v\left(w\right)v\left(z\right)\qquad\forall w,z\in\mathbb{R}^{2d}.
\]
\item $v$ satisfies the Gelfand-Raikov-Shilov (GRS) condition:
\[
\lim_{n\rightarrow\infty}v\left(nz\right)^{\frac{1}{n}}=1\qquad\forall z\in\mathbb{R}^{2d}.
\]
\end{enumerate}
Every weight of the form $v\left(z\right)=e^{a\left|z\right|^{b}}\left(1+\left|z\right|\right)^{s}\log^{r}\left(e+\left|z\right|\right)$,
with real parameters $a,r,s\ge0$ and $0\leq b<1$, is admissible.
Weights of particular interest are provided by
\begin{equation}\label{vs}
v_{s}\left(z\right)=\left\langle z\right\rangle ^{s}=\left(1+\left|z\right|^{2}\right)^{\frac{s}{2}},\qquad z\in\mathbb{R}^{2d},\,s\ge0.
\end{equation}
Observe that, for $s\geq 0$, the weight function  $v_s$ is equivalent to the submultiplicative weight $(1+|\cdot|)^s$, that is, there exist $C_1, C_2>0$ such that 
$$C_1 v_s(z)\leq (1+|z|)^s\leq C_2 v_s(z),\quad z\in\rdd.$$
A positive, even weight function $m$ on $\Renn$ is called  {\it
	$v$-moderate} if
$ m(z_1+z_2)\leq Cv(z_1)m(z_2)$  for all $z_1,z_2\in\Renn.$
We simply write $\mathcal{M}_{v}$ for the class of $v$-moderate weights.

Here and elsewhere the conjugate exponent $p'$ of $p \in [1,\infty]$ is defined by $1/p+1/p'=1$.
Moreover, in order to remain in the framework of tempered distributions, in what follows we shall consider weight functions $m$ on $\rd$ or $\rdd$ satisfying the following condition
\begin{equation}\label{M}
m(z)\geq 1,\quad \forall z\in \rd\quad \mbox{or} \quad m(z)\gtrsim \la z\ra^{-N}.
\end{equation}
for a suitable $N\in\bN$. The same holds for weights on $\rdd$.

\noindent
\textbf{Modulation Spaces.}
Given a non-zero window $g\in\sch(\Ren)$, a $v$-moderate weight
function $m$ on $\Renn$ satisfying \eqref{M}, and $1\leq p,q\leq
\infty$, the {\it
	modulation space} $M^{p,q}_m(\Ren)$ consists of all tempered
distributions $f\in\sch'(\Ren)$ such that $V_gf\in L^{p,q}_m(\Renn )$
(weighted mixed-norm space). The norm on $M^{p,q}_m$ is
$$
\|f\|_{M^{p,q}_m}=\|V_gf\|_{L^{p,q}_m}=\left(\int_{\Ren}
\left(\int_{\Ren}|V_gf(x,\o)|^pm(x,\o)^p\,
dx\right)^{q/p}d\o\right)^{1/q}  \, .
$$
If $p=q$, we write $M^p_m$ instead of $M^{p,p}_m$, and if $m(z)\equiv 1$ on $\Renn$, then we write $M^{p,q}$ and $M^p$ for $M^{p,q}_m$ and $M^{p,p}_m$.

Then  $\Mmpq (\Ren )$ is a Banach space
whose definition is independent of the choice of the window $g$. The class of admissible windows can be extended to $M^1_v$, as stated below (\cite[Thm.~11.3.7]{Grochenig_2006_Time}).
\begin{theorem} \label{admwind} Let
 $m$ be a  $v$-moderate weight  satisfying \eqref{M} and $g \in   M^1_{v} \setminus \{0\}$, then
$\|V_gf \|_{L^{p,q}_m}$ is an
equivalent norm for $M^{p,q}_m(\Ren)$.
\end{theorem}
Hence, given any $g \in M^1_v
(\rd )$ and $f\in\Mmpq $ we have the norm equivalence
\begin{equation}\label{normwind}
\|f\|_{\Mmpq } \asymp \|V_{g}f \|_{\Lmpq }.
\end{equation}
The previous equivalence will be heavily exploited in this work.
We  recall the inversion formula for
the STFT (see  \cite[Proposition 11.3.2]{Grochenig_2001_Foundations}): assume $g\in M^{1}_v(\rd)\setminus\{0\}$,
$f\in M^{p,q}_m(\rd)$, with $m$ satisfying \eqref{M} then
\begin{equation}\label{invformula}
f=\frac1{\|g\|_2^2}\int_{\R^{2d}} V_g f(z) \pi (z)  g\, dz \, ,
\end{equation}
and the  equality holds in $M^{p,q}_m(\rd)$.\par
The adjoint operator of $V_g$,  defined by
$$V_g^\ast F(t)=\intrdd F(z)  \pi (z) g dz \, ,
$$
maps the Banach space $L^{p,q}_m(\rdd)$ into $M^{p,q}_m(\rd)$. In particular, if $F=V_g f$ the inversion formula \eqref{invformula} reads
\begin{equation}\label{treduetre}
{\rm Id}_{M^{p,q}_m}=\frac 1 {\|g\|_2^2} V_g^\ast V_g.
\end{equation}

\noindent
\textbf{Wiener Amalgam Spaces.} Fix $g\in \cS(\rd) \setminus \left\{ 0 \right\} $. Consider \emph{even} weight functions $u,w$ on $\rd$ satisfying \eqref{M}. Then the Wiener amalgam space $W(\Fur L^p_u,L^q_w)(\rd)$ is the space of distributions $f\in\cS'(\rd)$ such that
\[
\|f\|_{W(\Fur L^p_u,L^q_w)(\rd)}:=\left(\int_{\Ren}
\left(\int_{\Ren}|V_gf(x,\o)|^p u^p(\o)\,
d\o\right)^{q/p} w^q(x)d x\right)^{1/q}<\infty  \,
\]
(obvious modifications for $p=\infty$ or $q=\infty$).
Using the fundamental identity of \tfa\, \eqref{FI}, we can write $|V_g f(x,\o)|=|V_{\hat g} \hat f(\o,-x)| = |\mathcal F (\hat f \, T_\o \overline{\hat g}) (-x)|$  and (recall $u(x)=u(-x)$)
$$
\| f \|_{{M}^{p,q}_{u\otimes w}} = \left( \int_{\rd} \| \hat f \ T_{\o} \overline{\hat g} \|_{\cF L^p_u}^q w^q(\o) \ d \o \right)^{1/q}
= \| \hat f \|_{W(\cF L_u^p,L_w^q)}.
$$
Hence the Wiener amalgam spaces under our consideration are simply the image under \ft\, of modulation spaces
\begin{equation}\label{W-M}
\cF ({M}^{p,q}_{u\otimes w})=W(\cF L_u^p,L_w^q).
\end{equation}
Using the relation \eqref{W-M} and Theorem \ref{admwind} one can easily infer the following issue.
\begin{theorem}\label{admwindW}
Let $w_i$ be a $v_i$-moderate weight on $\rd$, $i=1,2$ satisfying \eqref{M} and $g\in W(\cF L^1_{w_1},L^1_{w_2})$. Then $\|\| V_g f(x,\cdot)\|_{L^p_{w_2}}\|_{L^q_{w_1}}$ is an equivalent norm for $f\in W(\cF L^p_{w_1},L^q_{w_2})$. In other words, the class of admissible windows for $f\in W(\cF L^p_{w_1},L^q_{w_2})$ can be extended to $W(\cF L^1_{w_1},L^1_{w_2})$. 
\end{theorem}

\subsection{$\tau$-Pseudodifferential Operators}
In the spirit of the time-frequency-oriented presentation given in
\cite{Grochenig_2001_Foundations}, let us introduce the $\tau$-pseudodifferential operators
by means of superposition of time-frequency shifts, i.e.
\begin{equation}
\opt  \left(\sigma\right)f\left(x\right) = \int_{\mathbb{R}^{2d}}\hat{\sigma}\left(\omega,u\right)e^{-2 \pi i \left(1-\tau\right)\omega u}\left(T_{-u}M_{\omega}f\right)\left(x\right){d}u{d}\omega,\qquad x\in\mathbb{R}^{d},\label{opt tfs}
\end{equation}
for any $\tau\in\left[0,1\right]$. The symbol $\sigma$ and the function
$f$ belong to suitable function spaces, to be determined in order
for the previous expression to make sense. For instance, with obvious
modifications to \cite[Lem. 14.3.1]{Grochenig_2001_Foundations}, we immediately get that $\opt  \left(\sigma\right)$
maps $\mathcal{S}\left(\mathbb{R}^{d}\right)$ to $\mathcal{S}'\left(\mathbb{R}^{2d}\right)$
whenever $\sigma\in\mathcal{S}'\left(\mathbb{R}^{2d}\right)$. 

Notice that the representation \eqref{opt tfs} is a well-defined
absolutely convergent integral if $\hat{\sigma}\in L^{1}\left(\mathbb{R}^{2d}\right)$
and $f\in\mathcal{S}\left(\mathbb{R}^{d}\right)$. Under these assumptions,
easy computations allow to retrieve the usual integral form of $\tau$-pseudodifferential
operators, i.e.,
\begin{flalign*}
\opt  \left(\sigma\right)f\left(x\right) & =\int_{\mathbb{R}^{2d}}e^{2\pi i\left(x-y\right)\omega}\sigma\left(\left(1-\tau\right)x+\tau y,\omega\right)f\left(y\right){d}y{d}\omega.
\end{flalign*}

We finally aim to represent $\opt (\sigma)$ as an integral operator
of the form 
\[
\opt  \left(\sigma\right)f\left(x\right)=\int_{\mathbb{R}^{2d}}k\left(x,y\right)f\left(y\right){d}y.
\]
Let us introduce the operator $\mathfrak{T}_{\tau}$ acting on functions
on $\mathbb{R}^{2d}$ as
\[
\mathfrak{T}_{\tau}F\left(x,y\right) = F\left(x+\tau y,x-\left(1-\tau\right)y\right),\qquad\mathfrak{T}_{\tau}^{-1}F\left(x,y\right)=F\left(\left(1-\tau\right)x+\tau y,x-y\right),
\]
and denote by $\mathcal{F}_{i}$, $i=1,2$, the partial Fourier
transform with respect to the $i$-th $d-$dimensional variable (it is then clear that $\mathcal{F}=\mathcal{F}_{1}\mathcal{F}_{2}$).
Notice that 
\[
k\left(x,y\right)=\mathfrak{T}_{\tau}^{-1}\mathcal{F}_{2}^{-1}\sigma\left(x,y\right)=\mathcal{F}_{1}^{-1}\hat{\sigma}\left(\left(1-\tau\right)x+\tau y,y-x\right).
\]

Since the operators $\mathfrak{T}_{\tau}$ and $\mathcal{F}_{i}$
are continuous bijections on $\mathcal{S}\left(\mathbb{R}^{2d}\right)$,
the kernel $k$ is well-defined (as a tempered distribution) also
for symbols in $\mathcal{S}'\left(\mathbb{R}^{2d}\right)$ and we
can finally recover the representation by duality given in the Introduction.
We resume these remarks in the following result. 
\begin{prop}{\label{optker}}
	For any symbol $\sigma\in\mathcal{S}'\left(\mathbb{R}^{2d}\right)$
	and any real $\tau\in\left[0,1\right]$, the map $\opt (\sigma) \,:\,\mathcal{S}\left(\mathbb{R}^{d}\right)\rightarrow\mathcal{S}\left(\mathbb{R}^{d}\right)$
	is defined as integral operator with distributional kernel 
	\[
	k=\mathfrak{T}_{\tau}^{-1}\mathcal{F}_{2}^{-1}\sigma\in\mathcal{S}'\left(\mathbb{R}^{2d}\right),
	\]
	meaning that, for any $f,g\in\mathcal{S}\left(\mathbb{R}^{d}\right)$,
	\[
	\left\langle \opt  \left(\sigma\right)f,g\right\rangle =\left\langle k,g\otimes\overline{f}\right\rangle .
	\]
	In particular, since the representation 
	\[
	W_{\tau}\left(f,g\right)\left(x,\omega\right)=\mathcal{F}_{2}\mathfrak{T}_{\tau}\left(f\otimes\overline{g}\right)\left(x,\omega\right)
	\]
	holds for $f,g\in\mathcal{S}\left(\mathbb{R}^{d}\right)$, we have
	\[
	\left\langle \opt  \left(\sigma\right)f,g\right\rangle =\left\langle \sigma,W_{\tau}\left(g,f\right)\right\rangle .
	\]
\end{prop}

The issues discussed insofar can be summed up in the following claim, whose proof is an application of the celebrated Schwartz's kernel theorem (see for instance \cite[Theorem 14.3.4]{Grochenig_2001_Foundations}). 

\begin{theorem}\label{optreps}
Let $T\,:\,\mathcal{S}\left(\mathbb{R}^{d}\right)\rightarrow\mathcal{S}'\left(\mathbb{R}^{d}\right)$
be a continuous linear operator. There exist tempered distributions
$k,\sigma,F\in\mathcal{S}'\left(\mathbb{R}^{d}\right)$ and $\tau\in\left[0,1\right]$
such that $T$ admits the following representations:
\begin{enumerate}
	\item[$(i)$] as an integral operator: $\left\langle Tf,g\right\rangle =\left\langle k,g\otimes\overline{f}\right\rangle $
	for any $f,g\in\mathcal{S}\left(\mathbb{R}^{d}\right)$;
	\item[$(ii)$] as a $\tau$-pseudodifferential operator $T=\mathrm{Op}_{\tau}\left(\sigma\right)$
	with symbol $\sigma$;
	\item[$(iii)$] as a superposition (in a weak sense) of time-frequency shifts : $$T=\int_{\mathbb{R}^{2d}}F\left(x,\omega\right)e^{2\left(1-\tau\right)\pi ix\omega}T_{x}M_{\omega}\mathrm{d}x\mathrm{d}\omega.$$
\end{enumerate}
The relations among $k$, $\sigma$ and $F$ are the following:
\[
\sigma=\mathcal{F}_{2}\mathfrak{T}_{\tau}k,\qquad F=\mathcal{I}_{2}\hat{\sigma},
\]
where $\mathcal{I}_{2}$ denotes the reflection in the second $d$-dimensional
variable (i.e. $\mathcal{I}_{2}G\left(x,\omega\right)=G\left(x,-\omega\right)$,
$\left(x,\omega\right)\in\mathbb{R}^{2d}$). 
\end{theorem}

\section{Covariance property of $\tau$-Wigner distribution functions}

The proof of the following lemmas is a matter of computation. 
\begin{lemma}
For $\tau\in\left(0,1\right)$, we define the operator
\begin{equation}\label{Atau}
A_{\tau}\,:\,f\left(t\right)\mapsto\mathcal{I}f\left(\frac{1-\tau}{\tau}t\right),
\end{equation}
where $\mathcal{I}$ is the reflection operator $\left(\mathcal{I}g\left(t\right)=g\left(-t\right)\right)$.
Then, for any $z=(z_1,z_2)\in\mathbb{R}^{2d}$, 
\begin{equation}\label{e1}
\pi\left(z_1,z_2\right)A_{\tau}=A_{\tau}\pi\left(-\frac{1-\tau}{\tau}z_1,-\frac{\tau}{1-\tau}z_2\right),
\end{equation}
\begin{equation}\label{e2}
A_{\tau}\pi\left(z_1,z_2\right)=\pi\left(-\frac{\tau}{1-\tau}z_1,-\frac{1-\tau}{\tau} z_2\right)A_{\tau}.
\end{equation}
\end{lemma}

\begin{lemma}[{\cite[Lemma 6.2]{BoggiattoCubo2010}}]
The $\tau$-Wigner distribution admits a representation as short-time Fourier transform for $\tau\in\left(0,1\right)$:

\begin{equation}
W_{\tau}\left(f,g\right)\left(x,\omega\right)=\frac{1}{\tau^{d}}e^{2\pi i\frac{1}{\tau}x\omega}V_{A_{\tau}g}f\left(\frac{1}{1-\tau}x,\frac{1}{\tau}\omega\right).\label{eq:WDF STFT}
\end{equation}
\end{lemma}

Let us define the linear map $\mathcal{B}_{\tau}\,:\,\mathbb{R}^{2d}\rightarrow\mathbb{R}^{2d}$
as
\begin{equation}\label{BTlin}
\mathcal{B}_{\tau}\left(x\right)=\mathcal{B}_{\tau}\left(x_{1},x_{2}\right)=\left(\frac{1}{1-\tau}x_{1},\frac{1}{\tau}x_{2}\right)=\left(\begin{array}{cc}
\frac{1}{1-\tau}I_{d\times d} & 0_{d\times d}\\
0_{d\times d} & \frac{1}{\tau}I_{d\times d}
\end{array}\right)\left(\begin{array}{c}
x_{1}\\
x_{2}
\end{array}\right).
\end{equation}
Under this convention, Eq. (\ref{eq:WDF STFT}) becomes
\begin{equation}\label{e3}
W_{\tau}\left(f,g\right)\left(x,\omega\right)=\frac{1}{\tau^{d}}e^{2\pi i\frac{1}{\tau}x\omega}V_{A_{\tau}g}f\left(\mathcal{B}_{\tau}\left(x,\omega\right)\right).
\end{equation}
As it is customary, we use the symbol $\mathcal{B}_{\tau}$ also to
denote the matrix 
\begin{equation}\label{BT}
\mathcal{B}_{\tau}=\left(\begin{array}{cc}
\frac{1}{1-\tau}I_{d\times d} & 0_{d\times d}\\
0_{d\times d} & \frac{1}{\tau}I_{d\times d}
\end{array}\right).
\end{equation}

\begin{prop}\label{prop:covtauWDF}
For $\tau\in\left(0,1\right)$, $z=(z_1,z_2),$ $w=(w_1,w_2)\in\mathbb{R}^{2d}$, set 
\begin{equation}\label{Ttau}
\mathcal{T}_{\tau}\left(z,w\right)=\left(\begin{array}{c}
\left(1-\tau\right)z_{1}+\tau w_{1}\\
\tau z_{2}+\left(1-\tau\right)w_{2}
\end{array}\right)\quad z=(z_1,z_2),\,w=(w_1,w_2)\in\rdd.
\end{equation}
Then, for $f,g\in\mathcal{S}\left(\mathbb{R}^{d}\right)$ we have
\begin{equation}\label{CovtauWD}
W_{\tau}\left(\pi(z)f,\pi(w)g\right)\phas\!=\!c_\tau  M_{J(z-w)} T_{\mathcal{T}_{\tau}\left(z,w\right)} W_{\tau}\left(f,g\right)\phas
\end{equation}
where the phase factor $c_\tau$ is given by
\begin{equation}\label{ctau}
c_\tau=e^{2\pi i[(z_1-w_1)(\tau z_2+(1-\tau)w_2)]} .
\end{equation}
Equivalently, formula \eqref{CovtauWD} reads
\begin{equation}W_{\tau}\left(\pi\left(z\right)f,\pi\left(w\right)g\right)(u)
=c_{\tau}\cdot\pi\left(\mathcal{T}_{\tau}\left(z,w\right),J\left(z-w\right)\right)W_{\tau}\left(f,g\right)(u),\quad u\in\rdd,\label{CovtauWDcompact}
\end{equation}
where the phase factor $c_{\tau}$ is defined in \eqref{ctau}.
\end{prop}
\begin{proof}
Formula (\ref{eq:WDF STFT}) gives 
\[
W_{\tau}\left(\pi(z)f,\pi(w)g\right)\phas=\frac{1}{\tau^{d}}e^{2\pi i\frac{1}{\tau}x\omega}V_{A_{\tau}\pi\left(w\right)g}\left(\pi(z)f\right)\left(\frac{1}{1-\tau}x,\frac{1}{\tau}\omega\right).
\]
For what concerns the short-time Fourier transform above, we can
write, using \eqref{e2} and then \eqref{eq: composizione pi},
\begin{align*}
V_{A_{\tau}\pi(w)g}\left(\pi(z) f\right)&\left(\frac{1}{1-\tau}x,\frac{1}{\tau}\omega\right)=\left\langle \pi(z)f,\pi\left(\frac{1}{1-\tau}x,\frac{1}{\tau}\omega\right)A_{\tau}\pi(w)g\right\rangle \\
&=\left\langle \pi\left(z\right)f,\pi\left(\frac{1}{1-\tau}x,\frac{1}{\tau}\omega\right)\pi\left(-\frac{\tau}{1-\tau}w_1,-\frac{1-\tau}{\tau}w_2\right)A_{\tau}g\right\rangle\\
&=\left\langle \pi\left(z\right)f,e^{-2\pi i\left(\frac{1}{1-\tau}x\right)\left(-\frac{1-\tau}{\tau}w_2\right)}\pi\left(\frac{x-\tau w_1}{1-\tau},\frac{\omega-\left(1-\tau\right)w_2}{\tau}\right)A_{\tau}g\right\rangle \\
&=e^{-2\pi i\frac{1}{\tau}x w_2}\left\langle M_{z_2}T_{z_1}f,M_{\frac{\omega-\left(1-\tau\right)w_2}{\tau}}T_{\frac{x-\tau w_1}{1-\tau}}A_{\tau}g\right\rangle\\
& =e^{-2\pi i\frac{1}{\tau}x w_2}\left\langle f,T_{-z_1}M_{-z_2}M_{\frac{\omega-\left(1-\tau\right)w_2}{\tau}}T_{\frac{x-\tau w_1}{1-\tau}}A_{\tau}g\right\rangle \\
&=e^{-2\pi i\frac{1}{\tau}x w_2}\left\langle f,T_{-z_1}M_{\frac{\omega-\left(1-\tau\right)w_2-\tau w_1}{\tau}}T_{\frac{x-\tau w_1}{1-\tau}}A_{\tau}g\right\rangle .
\end{align*}
The commutation relations for TF-shifts \eqref{eq: fundid} give
\begin{align*}
V_{A_{\tau}\pi\left(w\right)g}\!\!&\left(\pi\left(z\right)f\right)\left(\frac{1}{1-\tau}x,\frac{1}{\tau}\omega\right)\\&=e^{-2\pi i\frac{1}{\tau}xw_2}\left\langle f,e^{-2\pi i\left(-z_1\right)\left(\frac{\omega-\left(1-\tau\right)w_2-\tau z_2}{\tau}\right)}M_{\frac{\omega-\left(1-\tau\right)w_2-\tau z_2}{\tau}}T_{\frac{x-\tau w_2-\left(1-\tau\right)z_1}{1-\tau}}A_{\tau}g\right\rangle \\
&=e^{-2\pi i\frac{1}{\tau}x w_2}e^{-2\pi i\frac{1}{\tau}z_1\left(\omega-\left(1-\tau\right)w_2-\tau z_2\right)}\left\langle f,M_{\frac{\omega-\left(1-\tau\right)w_2-\tau z_2}{\tau}}T_{\frac{x-\tau w_1-\left(1-\tau\right)z_1}{1-\tau}}A_{\tau}g\right\rangle \\
&=e^{-2\pi i\frac{1}{\tau}z w_2}e^{-2\pi i\frac{1}{\tau}z_1\left(\omega-\left(1-\tau\right) w_2-\tau z_2\right)}\\
&\qquad\qquad \times\,\,V_{A_{\tau}g}f\!\left(\!\frac{x-\left[\left(1-\tau\right)z_1+\tau w_1\right]}{1-\tau},\frac{\omega-\left[\tau z_2+\left(1-\tau\right)w_2\right]}{\tau}\!\right).
\end{align*}
Coming back to the original problem, we can write
\begin{align*}
W_{\tau}\left(\pi\left(z\right)f,\pi\left(w\right)g\right)\phas&=\frac{1}{\tau^{d}}e^{2\pi i\frac{1}{\tau}x \omega }e^{-2\pi i\frac{1}{\tau}x w_2}e^{-2\pi i\frac{1}{\tau}z_1\left(\omega-\left(1-\tau\right)w_2-\tau z_2\right)}\\
&\qquad\times\,V_{A_{\tau}g}f\!\left(\!\frac{x-\left[\left(1-\tau\right)z_1+\tau w_1\right]}{1-\tau},\frac{\omega-\left[\tau z_1+\left(1-\tau\right)w_2\right]}{\tau}\!\right).
\end{align*}
Working on the phase factors, we obtain 
\begin{align*}
W_{\tau}&\left(\pi(z)f,\pi(w)g\right)\phas\\
&=e^{-2\pi i\frac{1}{\tau}xw_2} e^{-2\pi i\frac{1}{\tau}z_1\left(\omega-\left(1-\tau\right)w_2-\tau z_2\right)} e^{2\pi i\frac{1}{\tau}x\left(\tau z_2+\left(1-\tau\right) w_2\right)}
e^{2\pi i\frac{1}{\tau}\omega\left(\left(1-\tau\right)z_1+\tau w_1\right)} \\
&\qquad\times e^{ -2\pi i\frac{1}{\tau}\left[\left(1-\tau\right)z_1+\tau w_1\right]\left[\tau z_2+\left(1-\tau\right)w_2\right]} T_{\mathcal{T}_\tau (z,w)}W_{\tau}\left(f,g\right)\phas\\
&=c_\tau M_{J(z-w)} T_{\mathcal{T}_\tau (z,w)}W_{\tau}\left(f,g\right)\phas,
\end{align*}
where $c_\tau$ is defined in \eqref{ctau}. Finally, formula \eqref{CovtauWDcompact} is an easy computation.
\end{proof}
\begin{remark}
(i) Notice that $\mathcal{T}_{\tau}\left(z,w\right)=\mathcal{T}_{1-\tau}\left(w,z\right)$ for every $z,w\in\mathbb{R}^{2d}$. Moreover,  $\mathcal{T}_\tau$ is well-defined for any  $\tau\in\left[0,1\right]$.\\
(ii) The proof of the covariance property heavily relies on the representation
(\ref{eq:WDF STFT}) and on the properties of the operator $A_{\tau}$,
which are both well-defined only when $\tau\in\left(0,1\right)$. We shall see in the sequel that the same formula \eqref{CovtauWDcompact} holds for the end-points $\tau=0$ and $\tau=1$.
\end{remark}

We now state the covariance property for the Rihaczek distribution:
as announced in the remark above,  we can
retrieve it by the naive substitution $\tau=0$ in \eqref{CovtauWD} or, equivalently, \eqref{CovtauWDcompact}.
\begin{prop} \label{prop:covariance 0} Consider  $f,g\in\mathcal{S}\left(\mathbb{R}^{d}\right)$. For $z=(z_1,z_2), w=(w_1,w_2)\in\mathbb{R}^{2d}$,
\begin{equation}\label{CovRD}
W_{0}\left(\pi(z)f,\pi(w)g\right)\phas=c_{0}e^{2\pi i[\phas J(z_1-w_1,z_2-w_2)]}W_{0}\left(f,g\right)\left(x-z_1,\omega-w_2\right),
\end{equation}
where $c_0=e^{2\pi i(z_1-w_1)w_2}$ is a phase factor. \par
Equivalently,  formula \eqref{CovRD} reads
\[
W_{0}\left(\pi\left(z\right)f,\pi\left(w\right)g\right)=c_{0}\cdot\pi\left(\left(z_{1},w_{2}\right),J\left(z-w\right)\right)W_{0}\left(f,g\right).
\]

\end{prop}
\begin{proof}
We have
\begin{flalign*}
W_{0}\left(\pi\left(z_1,z_2\right)f,\pi\left(w_1,w_2\right)g\right)\phas & =e^{-2\pi i x\omega}M_{z_2}T_{z_1}f\cdot\overline{\mathcal{F}\left(M_{w_2}T_{w_1}g\right)\left(\omega\right)}\\
 & =e^{-2\pi i x\omega}e^{2\pi ixz_2}f\left(x-z_1\right)\overline{T_{w_2}M_{-w_1}\hat{g}\left(\omega\right)}\\
 & =e^{-2\pi ix \omega}e^{2\pi ix z_2}e^{2\pi iw_1\left(\omega-w_2\right)}f\left(x-z_1\right)\overline{\hat{g}\left(\omega-w_2\right)}\\
 & =e^{2\pi i w_2\left(z_1-w_1\right)}e^{2\pi ix\left(z_2-w_2\right)}e^{-2\pi i\omega\left(z_1-w_1\right)}\\
 &\qquad\qquad \times W_{0}\left(f,g\right)\left(x-z_1,\omega-w_2\right),
\end{flalign*}
 as desired.
\end{proof}

We finally state the covariance property for the conjugate
Rihaczek distribution, whose proof follows almost at once from
the preceding result. It is again important to notice that  it can be inferred by the naive substitution
$\tau=1$ in \eqref{CovtauWD}.
\begin{prop} Consider $f,g\in\mathcal{S}\left(\mathbb{R}^{d}\right)$.
For $z=(z_1,z_2),\,w=(w_1,w_2)\in\mathbb{R}^{2d}$, we have
\begin{equation}\label{CovconjRD}
W_{1}\left(\pi(z)f,\pi(w)g\right)\phas =c_{1}e^{2\pi i[\phas J(z_1-w_1,z_2-w_2)]}W_{1}\left(f,g\right)\left(x-w_1,\omega-z_2\right),
\end{equation}
where  $c_{1}=e^{2\pi i(z_1-w_1) z_2}$ is a phase factor. Equivalently,  formula \eqref{CovconjRD} reads
\[
W_{1}\left(\pi\left(z\right)f,\pi\left(w\right)g\right)=c_{1}\cdot\pi\left(\left(w_{1},z_{2}\right),J\left(z-w\right)\right)W_{1}\left(f,g\right).
\]
\end{prop}
\begin{proof}
It follows by the covariance property for the Rihaczek distribution:
\begin{align*}
W_{1}\left(\pi(z)f,\pi(w)g\right)\phas&=\overline{W_{0}\left(\pi(w)g,\pi(z)f\right)}\phas\\
&=e^{-2\pi i(w_1-z_1)z_2} e^{2\pi i[\phas J(z_1-w_1,z_2-w_2)]}\overline{W_{0}\left(g,f\right)}\left(x-w_1,\omega-z_2\right)\\
&=e^{2\pi i(z_1-w_1) z_2} e^{2\pi i[\phas J(z_1-w_1,z_2-w_2)]}W_{1}\left(f,g\right)\left(x-w_1,\omega-z_2\right).
\end{align*}
This completes the proof.
\end{proof}

\section{Almost diagonalization}
Lemma 3.1 in \cite{Grochenig_2006_Time} has his  analogous for $\tau$-pseudodifferential operators as follows.
\begin{lemma}
\label{lem:STFT-gaborm}Fix a non-zero window $\varphi\in \cS(\rd)$
and set $\Phi_{\tau}=W_{\tau}\left(\varphi,\varphi\right)$ for $\tau\in\left[0,1\right]$.
Then, for $\sigma\in \cS'\left(\mathbb{R}^{2d}\right)$,
\begin{equation}
\left|\left\langle \opt \left(\sigma\right)\pi\left(z\right)\varphi,\pi\left(w\right)\varphi\right\rangle \right|=\left|{V}_{\Phi_{\tau}}\sigma\left(\mathcal{T}_{\tau}\left(z,w\right),J\left(w-z\right)\right)\right|=\left|{V}_{\Phi_{\tau}}\sigma\left(x,y\right)\right|\label{eq:gaborm as STFT}
\end{equation}
and
\begin{equation}
\left|{V}_{\Phi_{\tau}}\sigma\left(x,y\right)\right|=\left|\left\langle \opt \left(\sigma\right)\pi\left(z\left(x,y\right)\right)\varphi,\pi\left(w\left(x,y\right)\right)\varphi\right\rangle \right|,\label{eq:STFT as gaborm}
\end{equation}
for all $w,z,x,y\in\mathbb{R}^{2d}$, where $\mathcal{T}_\tau$ is defined in \eqref{Ttau} and
\begin{equation}
z\left(x,y\right)=\left(\begin{array}{c}
x_{1}+\left(1-\tau\right)y_{2}\\
x_{2}-\tau y_{1}
\end{array}\right),\qquad w\left(x,y\right)=\left(\begin{array}{c}
x_{1}-\tau y_{2}\\
x_{2}+\left(1-\tau\right)y_{1}
\end{array}\right).\label{eq:z,w}
\end{equation}
\end{lemma}
\begin{proof}
For $\tau\in (0,1)$, we use the time-frequency representation of $\tau$-Wigner pseudodifferential operators
and the covariance property in Prop. \ref{prop:covtauWDF}:
\begin{align*}
\left\langle \opt \left(\sigma\right)\pi\left(z\right)\varphi,\pi\left(w\right)\varphi\right\rangle &=\left\langle \sigma,W_{\tau}\left(\pi\left(w\right)\varphi,\pi\left(z\right)\varphi\right)\right\rangle \\
&=\left\langle \sigma,c_\tau M_{J\left(w-z\right)}T_{\mathcal{T}_{\tau}\left(w,z\right)}W_{\tau}\left(\varphi,\varphi\right)\right\rangle\\
&=\overline{c_\tau}\, {V}_{\Phi_{\tau}}\sigma\left(\mathcal{T}_{\tau}\left(w,z\right),J\left(w-z\right)\right).
\end{align*}

Formula \eqref{eq:STFT as gaborm} follows by setting $x=\mathcal{T}_{\tau}\left(w,z\right)$
and $y=J\left(w-z\right)$. 

For what concerns the case $\tau=0$, we use formula \eqref{CovRD} and calculate
\begin{align*}
\left\langle \mathrm{Op}_{0}\left(\sigma\right)\pi\left(z\right)\varphi,\pi\left(w\right)\varphi\right\rangle &=\left\langle \sigma,W_{0}\left(\pi\left(w\right)\varphi,\pi\left(z\right)\varphi\right)\right\rangle \\
&=\left\langle \sigma,c_{0}M_{J\left(w-z\right)}T_{\mathcal{T}_{0}\left(w,z\right)}W_{0}\left(\varphi,\varphi\right)\right\rangle\\
&=\overline{c_{0}}\, {V}_{\Phi_{0}}\sigma\left(\mathcal{T}_{0}\left(w,z\right),J\left(w-z\right)\right).
\end{align*}

Formula \eqref{eq:STFT as gaborm} follows by setting $x=\mathcal{T}_{0}\left(w,z\right)$
and $y=J\left(w-z\right)$. Solving the system for $z$ and $w$ gives
(\ref{eq:z,w}). The case $\tau=1$ follows the same pattern. Indeed, by \eqref{CovconjRD},
\begin{align*}
\left\langle \mathrm{Op}_{1}\left(\sigma\right)\pi\left(z\right)\varphi,\pi\left(w\right)\varphi\right\rangle &=\left\langle \sigma,W_{1}\left(\pi\left(w\right)\varphi,\pi\left(z\right)\varphi\right)\right\rangle \\
&=\left\langle \sigma,c_{1}M_{J\left(w-z\right)}T_{\mathcal{T}_{1}\left(w,z\right)}W_{1}\left(\varphi,\varphi\right)\right\rangle\\
&=\overline{c_{1}}\,{V}_{\Phi_{1}}\sigma\left(\mathcal{T}_{1}\left(w,z\right),J\left(w-z\right)\right).
\end{align*}
\end{proof}

\begin{lemma}\label{L1} Let $v$ be an admissible weight function on $\rdd$ and $1\leq p\leq\infty$. \\
	$(i)$ If $\tau\in [0,1]$,  $f\in M^1_v(\rd), g\in M^p_v(\rd)$ we have $W_\tau (g,f)\in M^{1,p}_{1\otimes v\circ J^{-1}}(\rdd)$, with
	\begin{equation}\label{ef1}
	\|W_\tau (g,f)\|_{M^{1,p}_{1\otimes v\circ J^{-1}}}\lesssim_\tau \|f\|_{ M^1_v}\|g\|_{ M^p_v}.
	\end{equation}
 $(ii)$ If $\tau\in (0,1)$,  $f \in M^1_{v\circ \cU_{\tau}}(\rd) , g \in M^p_v(\rd) $, we have $W_\tau (g,f)\in W\left(\cF L^1, L^p_{ v\circ \mathcal{B}_{\tau}}\right)(\rdd),$ with
\begin{equation}\label{ef2}
\|W_\tau (g,f)\|_{W\left(\cF L^1, L^p_{ v\circ \mathcal{B}_{\tau}}\right)} \lesssim_\tau \left[\tau\left(1-\tau\right)\right]^{-d/p'} \|f\|_{M^1_{v\circ \cU_{\tau}}} \|g\|_{M^p_v}, 
\end{equation}
and the matrices $\mathcal{B}_\tau$ and $\cU_{\tau}$ defined in \eqref{BT} and \eqref{UT}, respectively.
\end{lemma}
\begin{proof}
	Fix $\varphi_{1},\varphi_{2}\in\mathcal{S}\left(\mathbb{R}^{d}\right)$,
	thus $\Phi_{\tau}=W_{\tau}\left(\varphi_{1},\varphi_{2}\right)\in\mathcal{S}\left(\mathbb{R}^{d}\right)$.
	Therefore, by Lemma \ref{lem:STFT of tauWig} we can write (observe that the norm below may depend on $\tau$)
\begin{align*}
A&:=\left\Vert W_{\tau}\left(g,f\right)\right\Vert _{M_{1\otimes v\circ J^{-1}}^{1,p}\left(\mathbb{R}^{2d}\right)}=\left(\int_{\mathbb{R}^{2d}}\left(\int_{\mathbb{R}^{2d}}\left|{V}_{\Phi_{\tau}}W_\tau\left(g,f\right)\left(z,\zeta\right)\right|d z\right)^{p}v\left(J^{-1}\zeta\right)^{p}d \zeta\right)^{\frac{1}{p}}\\
&=\left(\int_{\mathbb{R}^{2d}}\left(\int_{\mathbb{R}^{2d}}\left|V_{\varphi_{1}}g\left(z-\mathcal{A}_{1-\tau}'\zeta\right)\right|\cdot\left|V_{\varphi_{2}}f\left(z+\mathcal{A}_{\tau}'\zeta\right)\right|d z\right)^{p}v\left(J^{-1}\zeta\right)^{p}d \zeta\right)^{\frac{1}{p}}.
\end{align*}
	The substitution $z'=z+\mathcal{A}_{\tau}'\zeta$, the properties
	of $\mathcal{A}_{\tau}'$ provided in Lemma \ref{lem:properties of Atau}
	(in particular $\mathcal{A}_{\tau}'+\mathcal{A}_{1-\tau}'=J$) and
	the fact that $v$ is an even function, give
	\begin{align*}
	A&
	=\left(\int_{\mathbb{R}^{2d}}\left(\int_{\mathbb{R}^{2d}}\left|V_{\varphi_{1}}g\left(z'-J\zeta\right)\right|\cdot\left|V_{\varphi_{2}}f\left(z'\right)\right|d z'\right)^{p}v\left(-J\zeta\right)^{p}d \zeta\right)^{\frac{1}{p}}\\
	&=\left(\int_{\mathbb{R}^{2d}}\left[\left|V_{\varphi_{2}}f\right|*\left|V_{\varphi_{1}}g\right|^{*}\left(J\zeta\right)\right]^{p}v\left(J\zeta\right)^{p}d \zeta\right)^{\frac{1}{p}}\\
&	=\left\Vert \left|V_{\varphi_{2}}f\right|*\left|V_{\varphi_{1}}g\right|^{*}\right\Vert _{L_{v}^{p}}\lesssim\left\Vert V_{\varphi_{2}}f\right\Vert _{L_{v}^{1}}\left\Vert V_{\varphi_{1}}g\right\Vert _{L_{v}^{p}}\asymp\left\Vert f\right\Vert _{M_{v}^{1}}\left\Vert g\right\Vert _{M_{v}^{p}}.
\end{align*}
	Recall that $V_{\f_1} g ^\ast(z)=\overline{V_{\f_1} g (-z)}$.
	In a similar fashion, 
\[
	B\coloneqq \| W_{\tau} \left(g,f\right)\| _{W\left(\mathcal{F}L^{1},L_{v\circ\mathcal{B}_{\tau}}^{p}\right)\left(\mathbb{R}^{2d}\right)}=\left(\int_{\mathbb{R}^{2d}}\left(\int_{\mathbb{R}^{2d}}\left|{V}_{\Phi_{\tau}}\left(g,f\right)\left(z,\zeta\right)\right|d \zeta\right)^{p}v\left(\mathcal{B}_{\tau}z\right)^{p}d z\right)^{\frac{1}{p}}\\
\]
{\small{}
	\[
	=\left[\int_{\mathbb{R}^{2d}}\left(\int_{\mathbb{R}^{2d}}\left|V_{\varphi_{1}}g\left(z-\sqrt{\tau\left(1-\tau\right)}\mathcal{A}_{1-\tau}\zeta\right)\right|\cdot\left|V_{\varphi_{2}}f\left(z+\sqrt{\tau\left(1-\tau\right)}\mathcal{A}_{\tau}\zeta\right)\right|\mathrm{d}\zeta\right)^{p}v\left(\mathcal{B}_{\tau}z\right)^{p}\mathrm{d}z\right]^{\frac{1}{p}}.
	\]
}
	Consider now the substitution $\eta=z-\sqrt{\tau\left(1-\tau\right)}\mathcal{A}_{1-\tau}\zeta$;
	 by the properties of $\mathcal{A}_{\tau}$ provided in
	Lemma \ref{lem:properties of Atau}, we can write
	\[
	B=\left[\tau\left(1-\tau\right)\right]^{-d}\left(\int_{\mathbb{R}^{2d}}\left(\int_{\mathbb{R}^{2d}}\left|V_{\varphi_{1}}g\left(\eta \right)\right|\cdot\left|V_{\varphi_{2}}f\left(\cB_{1-\tau}z+\cU_{1-\tau}\eta\right)\right|d \eta\right)^{p}v\left(\mathcal{B}_{\tau}z\right)^{p}d z\right)^{\frac{1}{p}}.
	\]
	Notice that $I-\cB_{\tau}=\cU_{\tau}$ and $\cU_{\tau}^{-1}=\cU_{1-\tau}$ (cf. \eqref{UT}), hence 
	\begin{align*}
B&=\left[\tau\left(1-\tau\right)\right]^{-d}\left(\int_{\mathbb{R}^{2d}}\left(\int_{\mathbb{R}^{2d}}\left|V_{\varphi_{1}}g\left(\eta \right)\right|\cdot\left|V_{\varphi_{2}}f\left(\cU_{1-\tau}(\eta + \cU_{\tau}\cB_{1-\tau}z\right)\right|d \eta\right)^{p}v\left(\mathcal{B}_{\tau}z\right)^{p}d z\right)^{\frac{1}{p}}\\
& =\left[\tau\left(1-\tau\right)\right]^{-d}\left(\int_{\mathbb{R}^{2d}}\left(\int_{\mathbb{R}^{2d}}\left|V_{\varphi_{1}}g\left(\eta \right)\right|\cdot\left|V_{\varphi_{2}}f\left(\cU_{1-\tau}(\eta - \cB_{\tau}z\right)\right|d \eta\right)^{p}v\left(\mathcal{B}_{\tau}z\right)^{p}d z\right)^{\frac{1}{p}} \\
&=\left[\tau\left(1-\tau\right)\right]^{-d}\left(\int_{\mathbb{R}^{2d}}(|V_{\varphi_{1}}g|*|V_{\varphi_{2}}f(\cU_{1-\tau}\cdot)|^{*})^{p}\left(\mathcal{B}_{\tau}z\right)v\left(\mathcal{B}_{\tau}z\right)^{p}d z\right)^{\frac{1}{p}}\\
	&=\left[\tau\left(1-\tau\right)\right]^{d\left(\frac{1}{p}-1\right)}\left\Vert |V_{\varphi_{1}}g|*|V_{\varphi_{2}}f(\cU_{1-\tau}\cdot)|^{*}\right\Vert _{L_{v}^{p}} \\&\lesssim \left[\tau\left(1-\tau\right)\right]^{d\left(\frac{1}{p}-1\right)}\left\Vert V_{\varphi_{2}}f\right\Vert _{L_{v\circ \cU_{\tau}}^{1}}\Vert V_{\varphi_{1}}g\Vert _{L_{v}^{p}} \\ &\asymp \left[\tau\left(1-\tau\right)\right]^{d\left(\frac{1}{p}-1\right)} \left\Vert f\right\Vert _{M_{v\circ \cU_{\tau}}^{1}} \left\Vert g\right\Vert _{M_{v}^{p}}.
	\end{align*}
\end{proof}

Thanks to this property we are able to extend to $\tau$-pseudodifferential operators,  $\tau\in [0,1]$, the characterization (obtained only for Weyl operators) in
\cite[Theorem 3.2]{Grochenig_2006_Time}. Namely, 
\begin{theorem}\label{teor41}
Let $v$ be an admissible weight function on $\mathbb{R}^{2d}$. Consider  $\varphi\in M^1_v\left(\mathbb{R}^{d}\right)\setminus\{0\}$ and a lattice $\Lambda \subseteq \mathbb{R}^{2d}$ such that  $\mathcal{G}\left(\varphi,\Lambda\right)$ is a  Gabor frame for $L^{2}\left(\mathbb{R}^{d}\right)$.
For any $\tau\in\left[0,1\right]$, the following properties are equivalent:
\begin{enumerate}
\item[$(i)$] $\sigma\in M_{1\otimes v\circ J^{-1}}^{\infty,1}\left(\mathbb{R}^{2d}\right)$.
\item[$(ii)$] $\sigma\in\mathcal{S}'\left(\mathbb{R}^{2d}\right)$ and there exists
a function $H_\tau\in L_{v}^{1}\left(\mathbb{R}^{2d}\right)$ such that
\begin{equation}
\left|\left\langle \opt \left(\sigma\right)\pi\left(z\right)\varphi,\pi\left(w\right)\varphi\right\rangle \right|\le H_\tau\left(w-z\right)\qquad\forall w,z\in\mathbb{R}^{2d}.\label{eq:almdiag J}
\end{equation}
	\item[$(iii)$] $\sigma\in\mathcal{S}'\left(\mathbb{R}^{2d}\right)$ and there exists
	a sequence $h_\tau\in\ell_{v}^{1}\left(\Lambda\right)$ such that
	\begin{equation}
	\left|\left\langle \opt \left(\sigma\right)\pi\left(\mu\right)\varphi,\pi\left(\lambda\right)\varphi\right\rangle \right|\le h_\tau\left(\lambda-\mu\right)\qquad\forall\lambda,\mu\in\Lambda.\label{eq:almdiag discr}
	\end{equation}
\end{enumerate}
\end{theorem}
\begin{proof}
The proof follows the pattern of the corresponding one for Weyl operators \cite[Theorem 3.2]{Grochenig_2006_Time}. We detail here only the case $(i)\Rightarrow (ii)$, to show where the $\tau$-dependence of the controlling function $H_\tau$ comes from. The case of the sequence $h_\tau$ is similar.
Consider $\f \in M^1_v(\rd)$ as in the assumptions and set $\Phi_\tau =W_\tau(\f,\f)\in M^{1}_{1\otimes v\circ J^{-1}}(\rdd)$, by Lemma \ref{L1}, part $(i)$.
This implies that the short-time Fourier transform $V_{\Phi_\tau} \sigma$  is well-defined for
$\sigma\in M_{1\otimes v\circ J^{-1}}^{\infty,1}\left(\mathbb{R}^{2d}\right)$ (cf. \cite[Theorem 11.3.7]{Grochenig_2001_Foundations}).
We now define
$ \tilde{H}_{\tau} (v)=\sup_{u\in\rdd} |V_{\Phi_\tau} \sigma(u,v)|$. By definition of
$M_{1\otimes v\circ J^{-1}}^{\infty,1}\left(\mathbb{R}^{2d}\right)$, we have $\tilde{H}_{\tau}\in L^1_{v\circ J^{-1}}(\rdd)$, so that Lemma \ref{lem:STFT-gaborm} implies
\begin{align*}
\left|\left\langle \opt \left(\sigma\right)\pi\left(z\right)\varphi,\pi\left(w\right)\varphi\right\rangle \right|&=\left|{V}_{\Phi_{\tau}}\sigma\left(\mathcal{T}_{\tau}\left(z,w\right),J\left(w-z\right)\right)\right|\\
&\leq \sup_{u\in\rdd} \left|{V}_{\Phi_{\tau}}\sigma\left(u,J\left(w-z\right)\right)\right| \\
& = \tilde{H}_{\tau} \left( J \left( w-z \right) \right) . 
\end{align*}

Setting $H_\tau =\tilde{H}_{\tau}\circ J$ we obtain the claim.
\end{proof}

\begin{corollary}\label{corteor41}
	Under the hypotheses of Theorem \ref{teor41}, let $\tau \in [0,1]$ be fixed and assume that
	$T\,:\,\mathcal{S}\left(\mathbb{R}^{d}\right)\rightarrow\mathcal{S}'\left(\mathbb{R}^{d}\right)$
	is continuous and satisfies one of the following conditions:
	\begin{itemize}
		\item[$(i)$] $\left|\left\langle T\pi\left(z\right)\varphi,\pi\left(w\right)\varphi\right\rangle \right|\le H\left(w-z\right)\quad\forall w,z\in\mathbb{R}^{2d}$
		for some $H\in L_{v}^{1}$. 
		\item[$(ii)$] $\left|\left\langle T\pi\left(\mu\right)\varphi,\pi\left(\lambda\right)\varphi\right\rangle \right|\le h\left(\lambda-\mu\right)\quad\forall\lambda,\mu\in\Lambda$
		for some $h\in\ell_{v}^{1}$. 
	\end{itemize}
	Therefore, $T=\opt \left(\sigma\right)$ for some symbol
	$\sigma\in M_{v\circ J^{-1}}^{\infty,1}\left(\mathbb{R}^{2d}\right)$. 
\end{corollary}

\begin{proof}
	Theorem \ref{optreps} implies that $T= \opt \left(\sigma \right)$ for some $\tau \in [0,1]$ and some symbol $\sigma \in \mathcal{S}' (\rdd)$. Then, the claim immediately follows from Theorem \ref{teor41}.
\end{proof}

For $\tau\in\left(0,1\right)$ (but also for extremal
values, see Theorem \ref{thm:almdiag weak} in the sequel), we are able to prove
a similar characterization  for symbols in the Wiener amalgam space $W\left(\mathcal{F}L^{\infty},L_{v}^{1}\right)$
via almost diagonalization. 

\begin{theorem}\label{thm:almost diag}
Let $v$ be an admissible weight function on $\mathbb{R}^{2d}$. Consider  $\varphi\in \cS(\rd)\setminus\{0\}$.
For any $\tau\in\left(0,1\right)$, the following properties are equivalent:
\begin{enumerate}
\item[$(i)$] $\sigma\in W\left(\mathcal{F}L^{\infty},L_{v\circ\mathcal{B}_{\tau}}^{1}\right)\left(\mathbb{R}^{2d}\right)$.
\item[$(ii)$] $\sigma\in\mathcal{S}'\left(\mathbb{R}^{2d}\right)$ and there exists
a function $H_\tau\in L_{v}^{1}\left(\mathbb{R}^{2d}\right)$ such that
\begin{equation}
\left|\left\langle \opt \left(\sigma\right)\pi\left(z\right)\varphi,\pi\left(w\right)\varphi\right\rangle \right|\le H_\tau\left(w-\mathcal{U}_{\tau}z\right)\qquad\forall w,z\in\mathbb{R}^{2d},\label{eq:almdiag}
\end{equation}
where the matrices $\mathcal{B}_{\tau} $ and  $\mathcal{U}_{\tau}$ are defined in \eqref{BT} and \eqref{UT}, respectively.
\end{enumerate}
\end{theorem}
\begin{proof}
$(i)\Rightarrow (ii)$. If $\f \in \cS(\rd)\subset M^1_v\left(\mathbb{R}^{d}\right) \cap M^1_{v\circ \cU_{\tau}(\rd)}$, $\f \neq 0$, then by Lemma \ref{L1}, part $(ii)$, with  $\Phi_\tau= W_\tau (\f,\f) \in W(\cF L^1, L^1_{v\circ \mathcal{B}_{\tau}})$. Take $\sigma\in W\left(\mathcal{F}L^{\infty},L_{v\circ\mathcal{B}_{\tau}}^{1}\right)$, then ${V}_{\Phi_{\tau}}\sigma$ is well defined  (cf. Theorem \ref{admwindW}) and
\[
\tilde{H}_\tau\left(x\right)=\sup_{y\in\mathbb{R}^{2d}}\left|{V}_{\Phi_{\tau}}\sigma\left(x,y\right)\right|\in L_{v\circ\mathcal{B}_{\tau}}^{1}\left(\mathbb{R}^{2d}\right).
\]
From Lemma \ref{lem:STFT-gaborm} we infer
\begin{flalign*}
\left|\left\langle \opt \left(\sigma\right)\pi\left(z\right)\varphi,\pi\left(w\right)\varphi\right\rangle \right| & =\left|{V}_{\Phi_{\tau}}\sigma\left(\mathcal{T}_{\tau}\left(w,z\right),J\left(w-z\right)\right)\right|\\
 & \le\sup_{y\in\mathbb{R}^{2d}}\left|{V}_{\Phi_{\tau}}\sigma\left(\mathcal{T}_{\tau}\left(w,z\right),y\right)\right|\\
 & =\tilde{H}_\tau\left(\mathcal{T}_{\tau}\left(w,z\right)\right).
\end{flalign*}
Notice that
\[
\mathcal{B}_{\tau}\left(\mathcal{T}_{\tau}\left(w,z\right)\right)=\left(\begin{array}{c}
w_{1}+\frac{\tau}{1-\tau}z_{1}\\
w_{2}+\frac{1-\tau}{\tau}z_{2}
\end{array}\right)=w-\mathcal{U}_{\tau}z,
\]
and thus $\tilde{H}_{\tau}\left(\mathcal{T}_{\tau}\left(w,z\right)\right)=\tilde{H}_{t}\left(\mathcal{B}_{\tau}^{-1}\left(w-\mathcal{U}_{\tau}z\right)\right)$. Define  $H_\tau= \tilde{H}_{\tau}\circ\mathcal{B}_{\tau}^{-1}$, then $H_\tau\in L_{v}^{1}(\rdd)$
since $\left\Vert H_\tau\right\Vert _{L_{v}^{1}}=\Vert \tilde{H}_{\tau}\circ\mathcal{B}_{\tau}^{-1}\Vert _{L_{v}^{1}}\asymp\Vert \tilde{H}_{\tau}\Vert _{L_{v\circ\mathcal{B}_{\tau}}^{1}}<\infty$. 

Let us now show $(ii) \Rightarrow (i)$. Assume that $\sigma\in\mathcal{S}'\left(\mathbb{R}^{2d}\right)$
and that $\opt (\sigma)$ is almost diagonalized by the
time-frequency shifts with dominating function $H_\tau\in L_{v}^{1}\left(\mathbb{R}^{2d}\right)$
as in (\ref{eq:almdiag}). Equation (\ref{eq:STFT as gaborm}) and
the results in the previous step (in particular, set $\tilde{H}_{\tau}= H_\tau\circ\mathcal{B}_{\tau}$
and then $H_\tau\left(w-\mathcal{U}_{\tau}z\right)=\tilde{H}_{\tau}\left(\mathcal{T}_{\tau}\left(w,z\right)\right)$
with $\mathcal{U}_{\tau}$ in \eqref{UT})
 allow to write 
\begin{flalign*}
\left|{V}_{\Phi_{\tau}}\sigma\left(x,y\right)\right| & =\left|\left\langle \opt \left(\sigma\right)\pi\left(z\left(x,y\right)\right)\varphi,\pi\left(w\left(x,y\right)\right)\varphi\right\rangle \right|\\
 & \leq H_\tau\left(w\left(x,y\right)-\mathcal{U}_{\tau}z\left(x,y\right)\right)\\
 & =\tilde{H}_{\tau}\left(\mathcal{T}_{\tau}\left(w\left(x,y\right),z\left(x,y\right)\right)\right)
\end{flalign*}
and since by construction $\mathcal{T}_{\tau}\left(w\left(x,y\right),z\left(x,y\right)\right)=x$,
we finally have 
\[
\left|{V}_{\Phi_{\tau}}\sigma\left(x,y\right)\right|\le \tilde{H}_{\tau}\left(x\right),\qquad\forall x\in\mathbb{R}^{2d}.
\]
Therefore, 
\begin{flalign*}
\left\Vert \sigma\right\Vert _{W\left(\mathcal{F}L^{\infty},L_{v\circ\mathcal{B}_{\tau}}^{1}\right)} & =\int_{\mathbb{R}^{2d}}\sup_{y\in\mathbb{R}^{2d}}\left|{V}_{\Phi_{\tau}}\sigma\left(x,y\right)\right|v\left(\mathcal{B}_{\tau}\left(x\right)\right)d x\\
 & \le\int_{\mathbb{R}^{2d}}\tilde{H}_{\tau}\left(x\right)v\left(\mathcal{B}_{\tau}\left(x\right)\right)d x\leq\Vert \tilde{H}_{\tau}\Vert _{L_{v\circ\mathcal{B}_{\tau}}^{1}}\asymp\Vert H_\tau\Vert _{L_{v}^{1}}<\infty
\end{flalign*}
and thus $\sigma\in W\left(\mathcal{F}L^{\infty},L_{v\circ\mathcal{B}_{\tau}}^{1}\right)$.
\end{proof}

\begin{example}\label{esempio}
	Consider $\sigma=\delta\in W\left(\mathcal{F}L^{\infty},L^{1}\right)(\rdd)$.  In this case, using formula \eqref{CovtauWD},
	\begin{align*}
	|\langle \mathrm{Op}_\tau(\delta) \pi\left(z\right)\varphi,\pi\left(w\right)\varphi \rangle | & =|\langle\delta, W_\tau(\pi\left(w\right)\varphi ,\pi\left(z\right)\varphi )\rangle|\\
	&= |\langle\delta, c_\tau M_{J(w-z)} T_{\mathcal{T}_{\tau}\left(w,z\right)} W_{\tau}\left(\f,\f\right)\ra|\\
	& = |T_{\mathcal{T}_{\tau}\left(w,z\right)} W_{\tau}\left(\f,\f\right)(0)|\\
	&=|W_{\tau}\left(\f,\f\right)(-\mathcal{T}_{\tau}\left(w,z\right))\\
	&=|W_{\tau}\left(\f,\f\right)(- \mathcal{B}_{\tau}^{-1}(w-\mathcal{U}_\tau z))|
	\end{align*}
	Choosing $H_\tau(z)= |W_{\tau}\left(\f,\f\right)(- \mathcal{B}_{\tau}^{-1}(z))|$ we obtain \eqref{eq:almdiag}, which  reduces to an equality in this case.
\end{example}

We remark that in this framework the discrete characterization of Theorem \ref{teor41} is lost. The main obstruction is the following: for a given lattice $\Lambda$, the inclusion  $\mathcal{U}_\tau \Lambda \subseteq \Lambda$ holds if and only if  $\mathcal{U}_\tau=\mathcal{U}_{1/2}=-I_{2d\times 2d}$, the (minus) identity matrix.  In this particular framework, the matrix $\mathcal{B}_{1/2}$ becomes
$\mathcal{B}_{1/2}=2 I_{2d\times 2d}$ and Theorem \ref{thm:almost diag} can be improved as follows.

\begin{corollary}\label{discreteWeyl}
Let $v$ be an admissible weight function on $\mathbb{R}^{2d}$ and set $v_2=v\circ 2I$. Consider  $\varphi\in \cS\left(\mathbb{R}^{d}\right)\setminus\{0\}$  such that  $\mathcal{G}\left(\varphi,\Lambda\right)$ is a Gabor frame for $L^{2}\left(\mathbb{R}^{d}\right)$.
For Weyl operators, the following properties are equivalent:
\begin{enumerate}
	\item[$(i)$] $\sigma\in W\left(\mathcal{F}L^{\infty},L_{v_2}^{1}\right)\left(\mathbb{R}^{2d}\right)$.
	\item[$(ii)$] $\sigma\in\mathcal{S}'\left(\mathbb{R}^{2d}\right)$ and there exists
	a function $H \in L_{v}^{1}\left(\mathbb{R}^{2d}\right)$ such that
	\begin{equation}
	\left|\left\langle \mathrm{Op}_{1/2}\left(\sigma\right)\pi\left(z\right)\varphi,\pi\left(w\right)\varphi\right\rangle \right|\le H \left(w+z\right)\qquad\forall w,z\in\mathbb{R}^{2d},\label{eq:almdiag weyl}
	\end{equation}
		\item[$(iii)$] $\sigma\in\mathcal{S}'\left(\mathbb{R}^{2d}\right)$ and there exists
		a sequence $h \in\ell_{v}^{1}\left(\Lambda\right)$ such that
		\begin{equation}
		\left|\left\langle \mathrm{Op}_{1/2}\left(\sigma\right)\pi\left(\mu\right)\varphi,\pi\left(\lambda\right)\varphi\right\rangle \right|\le h \left(\lambda+\mu\right)\qquad\forall\lambda,\mu\in\Lambda.\label{eq:almdiag discr2}
		\end{equation}
\end{enumerate}	
\end{corollary}
\begin{proof} The equivalence $(i)\Leftrightarrow (ii)$ is the consequence of Theorem \ref{thm:almost diag}, whereas the implications $(i)\Rightarrow (iii)$ and $(iii)\Rightarrow (i)$ follow the same pattern as in the proof of \cite{Grochenig_2006_Time}.
\end{proof}

Therefore, the symmetry of the Weyl quantization appears more powerful in this framework.
\begin{remark}
	Observe that the function space for  the window $\f$ in Theorem \ref{thm:almost diag} and Corollary \ref{discreteWeyl} can be extended from $\cS(\rd)\setminus\{0\}$ to $M^1_v(\rd)\cap M^1_{v\circ\mathcal{U}_\tau}(\rd)$, cf. Lemma \ref{L1}.
\end{remark}
The statement and the proof of Theorem \ref{thm:almost diag} are
not well suited for the degenerate cases $\tau=0$ and $\tau=1$ (notice
that $\mathcal{B}_{\tau}$ and $\mathcal{U}_{\tau}$ are not even
defined), thus the boundedness results which follow cannot be reproduced
in this context. However, a weaker version of that result can be proved
for any $\tau\in\left[0,1\right]$. 
\begin{theorem}
	\label{thm:almdiag weak} Consider an admissible weight $v$ on $\mathbb{R}^{2d}$ and
	a non-zero window $\varphi\in\mathcal{S}\left(\mathbb{R}^{d}\right)$.
	For any $\tau\in\left[0,1\right]$, the following properties are equivalent:
	\begin{enumerate}
		\item[$(i)$] $\sigma\in W\left(\mathcal{F}L^{\infty},L_{v}^{1}\right)\left(\mathbb{R}^{2d}\right)$.
		\item[$(ii)$] $\sigma\in\mathcal{S}'\left(\mathbb{R}^{2d}\right)$ and there exists
		a function $H_{\tau}\in L_{v}^{1}\left(\mathbb{R}^{2d}\right)$ such that
		\begin{equation}
		\left|\left\langle \opt \left(\sigma\right)\pi\left(z\right)\varphi,\pi\left(w\right)\varphi\right\rangle \right|\le H_{\tau}\left(\mathcal{T}_{\tau}\left(w,z\right)\right),\label{eq:almdiag weak-1}
		\end{equation}
	\end{enumerate}
	where $\mathcal{T}_{\tau}$ is defined in \eqref{Ttau}.
\end{theorem}
\begin{proof}
	$(i)\Rightarrow(ii)$. Take $\sigma\in W\left(\mathcal{F}L^{\infty},L_{v}^{1}\right)$
	and set 
	\[
	H_{\tau}\left(x\right)=\sup_{y\in\mathbb{R}^{2d}}\left|{V}_{\Phi_{\tau}}\sigma\left(x,y\right)\right|.
	\]
	By definition of $W\left(\mathcal{F}L^{\infty},L_{v}^{1}\right)$
	we have $H_{\tau}\in L_{v}^{1}\left(\mathbb{R}^{2d}\right)$. Since the
	linear map $\left(w,z\right)\mapsto J\left(w-z\right)$ is surjective,
	from Lemma \ref{lem:STFT-gaborm} we infer
	
	\begin{flalign*}
	\left|\left\langle \opt \left(\sigma\right)\pi\left(z\right)\varphi,\pi\left(w\right)\varphi\right\rangle \right| & =\left|{V}_{\Phi_{\tau}}\sigma\left(\mathcal{T}_{\tau}\left(w,z\right),J\left(w-z\right)\right)\right|\\
	& \le\sup_{y\in\mathbb{R}^{2d}}\left|{V}_{\Phi_{\tau}}\sigma\left(\mathcal{T}_{\tau}\left(w,z\right),y\right)\right|\\
	& =H_{\tau}\left(\mathcal{T}_{\tau}\left(w,z\right)\right).
	\end{flalign*}
	This gives the claim.\\
	$(ii)\Rightarrow (i)$. Assume that $\sigma\in\mathcal{S}'\left(\mathbb{R}^{2d}\right)$
	and that $\opt (\sigma)$ is almost diagonalized by the
	time-frequency shifts with dominating function $H_{\tau}\in L_{v}^{1}\left(\mathbb{R}^{2d}\right)$
	as in (\ref{eq:almdiag}). Equation (\ref{eq:STFT as gaborm}) and
	the results in the previous step allow to write 
	\begin{flalign*}
	\left|{V}_{\Phi_{\tau}}\sigma\left(x,y\right)\right| & =\left|\left\langle \opt \left(\sigma\right)\pi\left(z\left(x,y\right)\right)\varphi,\pi\left(w\left(x,y\right)\right)\varphi\right\rangle \right|\\
	& \le H_\tau\left(\mathcal{T}_{\tau}\left(w\left(x,y\right),z\left(x,y\right)\right)\right),
	\end{flalign*}
	and since by construction $\mathcal{T}_{\tau}\left(w\left(x,y\right),z\left(x,y\right)\right)=x$,
	we finally have 
	\[
	\left|{V}_{\Phi_{\tau}}\sigma\left(x,y\right)\right|\le H_{\tau}\left(x\right)\qquad\forall x\in\mathbb{R}^{2d}.
	\]
	Therefore, 
	\begin{flalign*}
	\left\Vert \sigma\right\Vert _{W\left(\mathcal{F}L^{\infty},L_{v}^{1}\right)} & =\int_{\mathbb{R}^{2d}}\sup_{y\in\mathbb{R}^{2d}}\left|{V}_{\Phi_{\tau}}\sigma\left(x,y\right)\right|v\left(x\right)d x\\
	& \le\int_{\mathbb{R}^{2d}}H_{\tau}\left(x\right)v\left(x\right)d x\leq\left\Vert H_{\tau}\right\Vert _{L_{v}^{1}}<\infty,
	\end{flalign*}
	that is $\sigma\in W\left(\mathcal{F}L^{\infty},L_{v}^{1}\right)$.
\end{proof}

\section{Boundedness results}
\subsection{Boundedness on Modulation Spaces}
As a consequence of the diagonalization provided by Theorem \ref{thm:almost diag}, 
we infer the boundedness of $\tau$-pseudodifferential
operators with symbols in the  Wiener amalgam space $W(\mathcal{F}L^{\infty},L_{v\circ\mathcal{B}_\tau}^{1})$ on every modulation
space, as follows.
\begin{theorem}
\label{thm:bounded optau} Fix $m\in\mathcal{M}_{v}$ satisfying \eqref{M}. For $\tau\in\left(0,1\right)$
consider a symbol $\sigma\in W(\mathcal{F}L^{\infty},L_{v\circ\mathcal{B}_{\tau}}^{1})\left(\mathbb{R}^{2d}\right)$, with the matrix $\mathcal{B}_{\tau}$ defined in \eqref{BT}. Then the operator $\opt (\sigma)$ is bounded from $M_{m}^{p,q}\left(\mathbb{R}^{d}\right)$
to $M_{m\circ\mathcal{U}_{\tau}}^{p,q}\left(\mathbb{R}^{d}\right)$,
$1\le p,q\le\infty$. 
\end{theorem}
\begin{proof} The proof uses the techniques developed in \cite[Theorem 3.3]{generalizedmetaplectic}.  We fix  $g(t)=e^{-\pi t^2}\in M^1_v(\rd)$ for every admissible weight $v$. Observe that  $\|g\|_2=1$, so that
	the inversion formula \eqref{treduetre} is simply $V_g^\ast V_g={\rm
		Id}$.  Writing  $T$ as
	$$T=V_g^\ast V_g T V_g^\ast V_g,$$
	then $V_g T V_g^\ast$ is an integral operator with kernel
	$$K_T(w,z)=\la T\pi(z)g,\pi(w)g\ra,\quad w,z\in\rdd \, .
	$$
	By definition, $V_g$ is bounded from $M^p_m(\rd )$ to
	$L^p_m(\rdd )$ and $V_g^*$ is bounded from $L^p_m(\rdd )$ to
	$M^p_m(\rd )$.
	Hence, if $V_g T V_g^\ast$ is bounded from $L^p_m(\rdd)$ to
	$L^p_{m\circ\mathcal{U}_{\tau}}(\rdd)$, then  $T$ is bounded from $M^p_m(\rd)$ to
	$M^p_{m\circ\mathcal{U}_{\tau}}(\rd)$.
	Observe that 
	$$ \mathcal{U}_{\tau}\circ \mathcal{B}_{1-\tau}= -\mathcal{B}_\tau
	$$
	so that $v\circ \mathcal{U}_{\tau}\circ \mathcal{B}_{1-\tau}=v\circ \mathcal{B}_\tau$, and recall that $\mathcal{U}_{1-\tau} ^{-1}=\mathcal{U}_{\tau}$. 
	Applying  Theorem \ref{thm:almost diag} with $1-\tau$ in place of $\tau$ and  with the admissible weight  $v\circ \mathcal{U}_{\tau}$ in place of $v$,  for  $F\in L^p_m(\rdd)$,
	\begin{align*}
	|V_g T V_g^\ast F(w)|& = \Big|\intrdd K_T(w,z) F(z) \, dz \Big| \leq \intrdd
	|F(z)| H_{1-\tau} (w- \mathcal{U}_{1-\tau} z) dz \\
	& =\intrdd |F(z)| (H_{1-\tau}\circ \mathcal{U}_{1-\tau})(\mathcal{U}_{\tau}w-z) dz\\
	&=F\ast(H_{1-\tau}\circ \mathcal{U}_{1-\tau})(\mathcal{U}_{\tau}w).
	\end{align*}
	By Theorem \ref{thm:almost diag},  $H_{1-\tau}$ is in $L^1_{v\circ \mathcal{U}_{\tau}}(\rdd)$ and thus $H_{1-\tau}\circ \mathcal{U}_{1-\tau} $ is   in $L^1_{v}(\rdd)$. Therefore, by Young's inequality
	 $F\ast(H_{1-\tau}\circ \mathcal{U}_{1-\tau})\in L^p_m(\rdd)\ast L^1_{v}(\rdd)\subset
	L^p_m(\rdd)$. This shows that   $V_g T V_g^\ast F\in
	L^p_{m\circ\mathcal{U}_{\tau}}(\rdd)$, as desired. 	
\end{proof}

	If we limit the study to the polynomial weights  $v_s$ defined in \eqref{vs}, then the theory of Fourier integral operators (FIOs) developed in \cite{generalizedmetaplectic} tells us further issues of Theorem \ref{thm:almost diag}. 
	
	We recall \cite[Definition 1.1]{generalizedmetaplectic}.
	
\begin{definition}
Given $\cA \in \mathrm{Sp}\left(2d,\mathbb{R}\right)$,
$g\in\cS(\rd)$,  and $s\geq0$, we say that  a
linear operator $T:\cS(\rd)\to\cS'(\rd)$ is in the
class $FIO(\cA,v_s)$ if there exists a function $H\in L^1_{v_s}(\rdd)$
such that the kernel of $T$ with respect to \tfs s satisfies the decay
condition
\begin{equation}\label{asterisco}
|\langle T \pi(z) g,\pi(w)g\rangle|\leq H(w-\cA z),\qquad \forall w,z\in\rdd.
\end{equation}
\end{definition} 

Moreover, the generalized metaplectic operators that can be represented as FIOs of type I are as follows (cf. \cite[Theorem 5.1]{generalizedmetaplectic}):
\begin{theorem}\label{rappresentazione}
	Let $s\geq 0$ and $\A=\begin{pmatrix}  A&B\\C&D\end{pmatrix} \in
	\mathrm{Sp}(2d,\R)$ with $ \det A\not=0$, and  let
	$T:\cS(\rd)\to\cS'(\rd)$ be a linear continuous operator.
	Then $T\in
	FIO(\mathcal{A},v_s)$ if and only if $T$ is a FIO of type I, i.e.
	\begin{equation}\label{fiotipo1}
	Tf(x)=\int_{\rd} e^{2\pi i \Phi\phas} \sigma\phas \hat{f}(\o)d\o
	\end{equation}
	with the  quadratic phase  $ \Phi(x,\omega)=\frac12  x CA^{-1}x+
	\omega  A^{-1} x-\frac12\eta  A^{-1}B\omega$ and a  symbol $\sigma\in M^{\infty,1}_{1\otimes v_s}(\rdd)$.
\end{theorem}

Observe that $v_s\circ \mathcal{B}_\tau\asymp v_s$ (with the bounds depending on $\tau$). If the symbol $\sigma$ is in 	$W(\cF L^\infty, L^1_{v_s})$ then Theorem \ref{thm:almost diag} says that \eqref{eq:almdiag} holds for a suitable function $H_{\tau}\in L^1_{v_s}$, so that the $\tau$-operator $\mathrm{Op}_\tau(\sigma)$ is in the
class $FIO(\mathcal{U}_{\tau},v_s)$. Moreover, the assumptions of Theorem \ref{rappresentazione}  are satisfied and we can thus represent $\opt (\sigma)$ as a type I FIO with phase $\Phi\phas=-\frac{1-\tau}{\tau} \omega x$ as follows:
\[
\opt (\sigma)f(x)=\int_{\mathbb{R}^{d}}e^{-2\pi i\frac{1-\tau}{\tau}\omega x}\rho\left(x,\omega\right)\hat{f}\left(\omega\right)d \omega,
\]
for a suitable symbol $\rho\in M_{1\otimes v_{s}}^{\infty,1}\left(\mathbb{R}^{d}\right)$. 
Finally, by Theorem 1.2 in \cite{generalizedmetaplectic} we obtain: 
\begin{corollary}
If  $\sigma \in W(\cF L^\infty, L^1_{v_s})$, $s\geq 0$, then the operator $\opt (\sigma)$ is bounded on every modulation space $M^p_{v_s}(\rd)$, for $1\leq p\leq \infty$ and $\tau \in (0,1)$.
\end{corollary}

\subsection{Boundedness on Wiener Amalgam Spaces}
We now turn to consider $\tau$-pseudodifferential operators and their boundedness on
Wiener amalgam spaces. We need two preliminary results. 
\begin{lemma}
\label{lem:symplectic covariance tau}For any $\sigma\in\mathcal{S}'\left(\mathbb{R}^{2d}\right)$
and $\tau\in\left[0,1\right]$, 
\[
\mathcal{F}\opt \left(\sigma\right)\mathcal{F}^{-1}=\mathrm{Op}_{1-\tau}\left(\sigma\circ J^{-1}\right).
\]
\end{lemma}
\begin{proof}
It is a particular case of the symplectic covariance property
of Shubin calculus, see \cite[Proposition 10]{dgTRAN2013}. We also refer to the books \cite{dgbook2016,horm3,WongWeylTransform1998}. 

\end{proof}
The proof of the following lemma is a straightforward computation. 
\begin{lemma}
\label{lem: STFT con J} For any non-zero window $G\in\mathcal{S}\left(\mathbb{R}^{2d}\right)$
and $\sigma\in\mathcal{S}'\left(\mathbb{R}^{2d}\right)$, then 
\[
{V}_{G}\left(\sigma\circ J\right)\left(z,\zeta\right)=\left({V}_{G\circ J^{-1}}\sigma\right)\left(Jz,J\zeta\right).
\]
Therefore, for any $1\le p,q\le\infty$, weights $u,v$ on $\mathbb{R}^{2d}$ and $\tau \in [0,1]$:
\begin{enumerate}
\item[$(i)$] $\sigma_{J}=\sigma\circ J\in  M_{(u\circ J\inv)\otimes (v\circ J\inv)}^{p,q}\left(\mathbb{R}^{2d}\right)$
if and only if $\sigma\in M_{u\otimes v}^{p,q}\left(\mathbb{R}^{2d}\right)$.
In particular, 
\[
\sigma\in M_{1\otimes v}^{\infty,1}\left(\mathbb{R}^{2d}\right) \Leftrightarrow\sigma_J\in M_{1\otimes\left(v\circ J\inv\right)}^{\infty,1}\left(\mathbb{R}^{2d}\right).
\]
\item[$(ii)$] $\sigma\in W\left(\mathcal{F}L_{u\circ J^{-1}}^{p},L_{v\circ J^{-1}}^{q}\right)$ if and only if $\sigma_{J}\in W\left(\mathcal{F}L_{u}^{p},L_{v}^{q}\right)$. In particular, 
\[
\sigma\in W\left(\mathcal{F}L^{\infty},L_{v\circ\mathcal{B}_{\tau}\circ J^{-1}}^{1}\right)\Leftrightarrow\sigma_{J}\in W\left(\mathcal{F}L^{\infty},L_{v\circ\mathcal{B}_{\tau}}^{1}\right).
\]
\end{enumerate}
\end{lemma}
\begin{proof} A direct computation shows that
\begin{align*}
{V}_{G}\sigma_{J}\left(z,\zeta\right)&=\left\langle \sigma_{J},M_{\zeta}T_{z}G\right\rangle =\int e^{-2\pi i\zeta\left(t,\omega\right)}\sigma_{J}\left(t,\omega\right)\overline{G\left(t-z_{1},\omega-z_{2}\right)}d td \omega\\&=
\int e^{-2\pi i\zeta\left(t,\omega\right)}\sigma\left(J\left(t,\omega\right)\right)\overline{G\left(t-z_{1},\omega-z_{2}\right)}d td \omega.
\end{align*}
With the substitution $\left(t',\omega'\right)=J\left(t,\omega\right)$,
we have
\begin{flalign*}
{V}_{G}\sigma_{J}\left(z,\zeta\right) & =\int e^{-2\pi i\zeta\cdot J^{-1}\left(t',\omega'\right)}\sigma\left(t',\omega'\right)\overline{G\left(J^{-1}\left[\left(t',\omega'\right)-Jz\right]\right)}d t'd \omega'\\
 & =\left\langle \sigma,M_{J\zeta}T_{Jz}\left(G\circ J^{-1}\right)\right\rangle =\left({V}_{G\circ J^{-1}}\sigma\right)\left(Jz,J\zeta\right).
\end{flalign*}
Therefore, 
\begin{flalign*}
\left\Vert \sigma_{J}\right\Vert _{M_{\left(u\circ J^{-1}\right)\otimes\left(v\circ J^{-1}\right)}^{p,q}} & =\left(\int_{\mathbb{R}^{2d}}\left(\int_{\mathbb{R}^{2d}}\left|V_{G}\sigma_{J}\left(z,\zeta\right)\right|^{p}u\left(J^{-1}z\right)^{p}dz\right)^{\frac{q}{p}}v\left(J^{-1}\zeta\right)^{q}d\zeta\right)^{\frac{1}{q}}\\
& =\left(\int_{\mathbb{R}^{2d}}\left(\int_{\mathbb{R}^{2d}}\left|V_{G\circ J^{-1}}\sigma\left(Jz,J\zeta\right)\right|^{p}u\left(J^{-1}z\right)^{p}dz\right)^{\frac{q}{p}}v\left(J^{-1}\zeta\right)^{q}d\zeta\right)^{\frac{1}{q}}\\
& =\left(\int_{\mathbb{R}^{2d}}\left(\int_{\mathbb{R}^{2d}}\left|V_{G\circ J^{-1}}\sigma\left(z,\zeta\right)\right|^{p}u\left(-z\right)^{p}dz\right)^{\frac{q}{p}}v\left(-\zeta\right)^{q}d\zeta\right)^{\frac{1}{q}}\\
& \asymp\left\Vert \sigma\right\Vert _{M_{u\otimes v}^{p,q}}.
\end{flalign*}

where we used the even property of the weight functions.  In a similar fashion, \begin{flalign*}
\left\Vert \sigma_{J}\right\Vert _{W\left(\mathcal{F}L_{u}^{p},L_{v}^{q}\right)}	&=\left(\int_{\mathbb{R}^{2d}}\left(\int_{\mathbb{R}^{2d}}\left|V_{G}\sigma_{J}\left(z,\zeta\right)\right|^{p}u\left(\zeta\right)^{p}d\zeta\right)^{\frac{q}{p}}v\left(z\right)^{q}dz\right)^{\frac{1}{q}} \\ 
& =\left(\int_{\mathbb{R}^{2d}}\left(\int_{\mathbb{R}^{2d}}\left|V_{G\circ J^{-1}}\sigma\left(z,\zeta\right)\right|^{p}u\left(J^{-1}z\right)^{p}dz\right)^{\frac{q}{p}}v\left(J^{-1}\zeta\right)^{q}d\zeta\right)^{\frac{1}{q}} \\
& \asymp\left\Vert \sigma\right\Vert _{W\left(\mathcal{F}L_{u\circ J^{-1}}^{p},L_{v\circ J^{-1}}^{q}\right)}.
\end{flalign*}
\end{proof}

Another ingredient is the boundedness of $\tau$-pseudodifferential  operators on modulation spaces, cf. Theorem 4.3 and Remark 4.5 in \cite{Toftweight2004} (see also \cite[Theorem 4.1]{Grochenig_2006_Time} for Weyl operators). 
\begin{theorem}\label{TW}
Consider  $m\in\mathcal{M}_{v}\left(\mathbb{R}^{2d}\right)$ satisfying \eqref{M}. For any $\tau\in [0,1]$ and $\sigma\in M^{\infty,1}_{1\otimes v\circ J\inv}$ the operator $\opt (\sigma)$ is bounded on $M^{p,q}_m(\rd)$, and there exists a constant $C_\tau>0$ such that
\begin{equation}\label{Toftweight}
\| \opt (\sigma)\|_{M^{p,q}_m}\leq C_\tau \|\sigma\|_{M^{\infty,1}_{1\otimes v\circ J\inv}}.
\end{equation}
\end{theorem}
We can now state the boundedness result for $\tau$-pseudodifferential  operators on Wiener amalgam spaces.
\begin{theorem}
Consider  $m=m_{1}\otimes m_{2}\in\mathcal{M}_{v}\left(\mathbb{R}^{2d}\right)$ satisfying \eqref{M}.
For any $\tau\in [0,1]$ and $\sigma\in M_{1\otimes v }^{\infty,1}\left(\mathbb{R}^{2d}\right)$,
the operator $\opt (\sigma)$ is bounded on $W\left(\mathcal{F}L_{m_{1}}^{p},L_{m_{2}}^{q}\right)\left(\mathbb{R}^{d}\right)$
with
$$
\| \opt (\sigma)\|_{W\left(\mathcal{F}L_{m_{1}}^{p},L_{m_{2}}^{q}\right)}\leq C_\tau \|\sigma\|_{M_{1\otimes v}^{\infty,1}},$$
for a suitable $C_\tau>0$.
\end{theorem}
\begin{proof}
Consider the following commutative diagram:
\[
\xymatrix{M_{m}^{p,q}\left(\mathbb{R}^{d}\right)\ar[r]^{\mathrm{Op}_{1-\tau}(\sigma_{J})} &  M_{m}^{p,q}\left(\mathbb{R}^{d}\right)\ar[d]^{\mathcal{F}}\\
W\left(\mathcal{F}L_{m_{1}}^{p},L_{m_{2}}^{q}\right)\left(\mathbb{R}^{d}\right)\ar[r]^{\mathrm{Op}_{\tau}(\sigma)}\ar[u]_{\mathcal{F}^{-1}} &  W\left(\mathcal{F}L_{m_{1}}^{p},L_{m_{2}}^{q}\right)\left(\mathbb{R}^{d}\right)
}
\]
Indeed, since  $\sigma\in M_{1\otimes v}^{\infty,1}\left(\mathbb{R}^{2d}\right)$,
$\sigma_{J}\in  M_{1\otimes (v\circ J\inv)}^{\infty,1}\left(\mathbb{R}^{2d}\right)$
by Lemma \ref{lem: STFT con J}. The operator $\mathrm{Op}_{1-\tau}\left(\sigma_J\right)$ is bounded on $M_{m}^{p,q}(\rd)$ by virtue of Theorem \ref{TW} with $\tau'=1-\tau\in [0,1]$ 
and the thesis follows at once thanks to the Lemma \ref{lem:symplectic covariance tau}.
\end{proof}

The same argument (with obvious modifications) allow to extend the boundedness result for \tpsdo s contained in Theorem \ref{thm:bounded optau} to Wiener amalgam spaces for symbols in suitable Wiener amalgam spaces.
\begin{theorem}
	Consider  $m=m_{1}\otimes m_{2}\in\mathcal{M}_{v}\left(\mathbb{R}^{2d}\right)$ satisfying \eqref{M}. For any $\tau\in\left(0,1\right)$ and $\sigma\in W\left(\mathcal{F}L^{\infty},L_{v\circ\mathcal{B}_{\tau}\circ J^{-1}}^{1}\right)\left(\mathbb{R}^{2d}\right)$, the operator $\mathrm{Op}_{\tau}\sigma$ is bounded from $W\left(\mathcal{F}L_{m_{1}}^{p},L_{m_{2}}^{q}\right)\left(\mathbb{R}^{d}\right)$ to $W\left(\mathcal{F}L_{m_{1}\circ\left(\mathcal{U}_{\tau}\right)_{1}}^{p},L_{m_{2}\circ\left(\mathcal{U}_{\tau}\right)_{2}}^{q}\right)\left(\mathbb{R}^{d}\right)$, $1\le p,q\le\infty$, where 
	\[ \left(\mathcal{U}_{\tau}\right)_{1}\left(x\right)=-\frac{\tau}{1-\tau}x,\qquad\left(\mathcal{U}_{\tau}\right)_{2}\left(x\right)=-\frac{1-\tau}{\tau}x,\qquad x\in\mathbb{R}^{d}.
	\]
\end{theorem} 

\begin{proof}
	Consider the following commutative diagram:
	\[
	\xymatrix{M_{m}^{p,q}\left(\mathbb{R}^{d}\right)\ar[r]^{\mathrm{Op}_{1-\tau}(\sigma_{J})} & M_{m\circ\mathcal{U}_{\tau}}^{p,q}\left(\mathbb{R}^{d}\right)\ar[d]^{\mathcal{F}}\\
		 W\left(\mathcal{F}L_{m_{1}}^{p},L_{m_{2}}^{q}\right)\left(\mathbb{R}^{d}\right)\ar[r]^(.395){\mathrm{Op}_{\tau}(\sigma)}\ar[r]\ar[u]_{\mathcal{F}^{-1}} &   W\left( \mathcal{F}L_{m_{1}\circ\left(\mathcal{U}_{\tau}\right)_{1}}^{p},L_{m_{2}\circ (\mathcal{U}_{\tau})_{2}}^{q}\right)(\mathbb{R}^{d})}
	\]
	Indeed, since $\sigma\in W\left(\mathcal{F}L^{\infty},L_{v\circ\mathcal{B}_{\tau}\circ J^{-1}}^{1}\right)\left(\mathbb{R}^{2d}\right)$,
	$\sigma_{J}\in W\left(\mathcal{F}L^{\infty},L_{v\circ\mathcal{B}_{\tau}}^{1}\right)\left(\mathbb{R}^{2d}\right)$
	by Lemma \ref{lem: STFT con J}. The operator $\mathrm{Op}_{1-\tau}\sigma_{J}$
	is bounded by virtue of Theorem \ref{thm:bounded optau} with $\tau'=1-\tau\in\left(0,1\right)$
	and the thesis follows at once thanks to the Lemma \ref{lem:symplectic covariance tau}.
\end{proof}

\subsection*{The cases $\tau=0$ and $\tau=1$}
Theorem \ref{thm:almdiag weak}  allows us to obtain  some boundedness results for  $\tau$-pseudodifferential operators with $\tau=0$ or $\tau=1$, having symbols in Wiener amalgam spaces.
	\begin{proposition}\label{Prop5.8}
		Assume $\sigma \in W(\cF L^\infty, L^1)(\rd)$.  Then the  Kohn-Nirenberg operator $\mathrm{Op_{KN}}(\sigma) \,\, (\tau = 0)$ is bounded on $M^{1,\infty}(\rd)$.
	\end{proposition}
	\begin{proof}
	Consider $H_0\left(\mathcal{T}_{0}\left(w,z\right)\right)$
with $H_0\in L^{1}\left(\mathbb{R}^{2d}\right)$ as in Theorem \ref{thm:almdiag weak}.
	The integral operator $T_{H_0}$ with kernel $H_0\left(\mathcal{T}_{0}\left(w,z\right)\right)=H_0\left(w_{1},z_{2}\right)$
	can be written as
	\[
	T_{H_0}F\left(w\right)=\int_{\mathbb{R}^{2d}}H_0\circ \mathcal{T}_{0}\left(w,z\right)F\left(z\right)d z=\int_{\mathbb{R}^{d}}\int_{\mathbb{R}^{d}}H\left(w_{1},z_{2}\right)F\left(z_{1},z_{2}\right)d z_{1}d z_{2}.
	\]
	It is immediate to notice that $T_{H_0}\,:\,L^{1,\infty}\left(\mathbb{R}^{2d}\right)\rightarrow L^{1,\infty}\left(\mathbb{R}^{2d}\right)$
	is a bounded operator. Then, for a fixed non-zero window $g\in\mathcal{S}\left(\mathbb{R}^{d}\right)$,
	we have that 
	\[
	T=V_{g}^{*}T_{H_0}V_{g}\,:\,M^{1,\infty}\left(\mathbb{R}^{d}\right)\rightarrow M^{1,\infty}(\rd)
	\]
	is a bounded operator. The claim then follows. 
\end{proof}
\begin{proposition} \label{Prop5.9}
		Assume $\sigma \in W(\cF L^\infty, L^1)(\rd)$. Then the operator \textquotedblleft with right symbol\textquotedblright \, $\mathrm{Op}_{1}(\sigma) \,\, (\tau = 1)$ is bounded on $W(\cF L^1, L^\infty)(\rd)$. 
\end{proposition}		
\begin{proof}
		Again, we apply Theorem \ref{thm:almdiag weak} and  consider $H_1\left(\mathcal{T}_{1}\left(w,z\right)\right)$ with $H_1\in L^{1}\left(\mathbb{R}^{2d}\right)$.
	The integral operator $T_{H}$ with kernel $H_1\left(\mathcal{T}_{1}\left(w,z\right)\right)=H_1\left(z_{1},w_{2}\right)$
	can be written as
	\[
	T_{H_1}F\left(w\right)=\int_{\mathbb{R}^{2d}}H_1\circ \mathcal{T}_{1}\left(w,z\right)F\left(z\right)d z=\int_{\mathbb{R}^{d}}\int_{\mathbb{R}^{d}}H_1\left(z_{1},w_{2}\right)F\left(z_{1},z_{2}\right)d z_{1}d z_{2}.
	\]
	It is immediate to notice that $T_{H_1}\,:\,L_{z_{1}}^{\infty}\left(L_{z_{2}}^{1}\right)\left(\mathbb{R}^{2d}\right)\rightarrow L_{w_{1}}^{\infty}\left(L_{w_{2}}^{1}\right)\left(\mathbb{R}^{2d}\right)$
	is a bounded operator. Then, for a fixed non-zero window $g\in\mathcal{S}\left(\mathbb{R}^{d}\right)$,
	we have that 
	\[
	T=V_{g}^{*}T_{H_1}V_{g}\,:\,W\left(\mathcal{F}L^{1},L^{\infty}\right)\left(\mathbb{R}^{d}\right)\rightarrow W\left(\mathcal{F}L^{1},L^{\infty}\right)\left(\mathbb{R}^{d}\right)
	\]
	is a bounded operator. This concludes the proof. 
\end{proof}

The consequences of the almost diagonalization of $\tau$-operators are manifold. We notice that the  results of this section can be extended by interpolation to symbols in $W(\cF L^p, L^q)(\rdd)$, following the pattern of \cite{CNIMRN2018}. This subject with be further investigated in a subsequent paper.

\section{Algebra and Wiener properties}
The connection with the theory Fourier integral operators established
in the previous section allows to investigate further properties of
$\tau$-operators. First of all, notice that for
any $\tau_{1},\tau_{2}\in\left(0,1\right)$,
\[
\mathcal{U}_{\tau_{1}}\mathcal{U}_{\tau_{2}}=\left(\begin{array}{cc}
\frac{\tau_{1}\tau_{2}}{\left(1-\tau_{1}\right)\left(1-\tau_{2}\right)}I_{d\times d} & 0_{d\times d}\\
0_{d\times d} & \frac{\left(1-\tau_{1}\right)\left(1-\tau_{2}\right)}{\tau_{1}\tau_{2}}I_{d\times d}
\end{array}\right).
\]
In particular, 
\[
\mathcal{U}_{\tau}\mathcal{U}_{1-\tau}=\mathcal{U}_{1-\tau}\mathcal{U}_{\tau}=I_{2d\times2d}.
\]

Therefore, composition properties of operators in the class $FIO\left(\mathcal{A},v_{s}\right)$
(see \cite[Theorems 3.4]{generalizedmetaplectic} and  Theorem \ref{teor41}) yield the following result. 
\begin{theorem}[Algebra property]
	For any $a,b\in W\left(\mathcal{F}L^{\infty},L_{v_{s}}^{1}\right)$
	and $\tau\in\left(0,1\right)$, there exists a symbol $c\in M_{1\otimes v_{s}}^{\infty,1}$
	such that
	\[
\mathrm{Op}_{\tau}\left(a\right)\mathrm{Op}_{1-\tau}\left(b\right)=\mathrm{Op}_{1/2}\left(c\right).
	\]
\end{theorem}
\begin{remark}
	On the other hand, for any choice of $\tau_{1},\tau_{2}\in\left(0,1\right)$,
	there is no $\tau\in\left(0,1\right)$ such that $\mathcal{U}_{\tau_{1}}\mathcal{U}_{\tau_{2}}=\mathcal{U}_{\tau}$.
	This immediately implies that there is no $\tau$-quantization rule
	such that composition of $\tau$-operators with symbols in $W\left(\mathcal{F}L^{\infty},L_{v_{s}}^{1}\right)$
	has symbol in the same class. In a similar fashion, given $a\in W\left(\mathcal{F}L^{\infty},L_{v_{s}}^{1}\right)$,
	$b\in M_{1\otimes v_{s}}^{\infty,1}$ and $\tau,\tau_{0}\in\left(0,1\right)$,
	we have 
	\[
	\mathrm{Op}_{\tau_{0}}\left(b\right)\mathrm{Op}_{\tau}\left(a\right)=\mathrm{Op}_{\tau}\left(c_{1}\right),\qquad\mathrm{Op}_{\tau}\left(a\right)\mathrm{Op}_{\tau_{0}}\left(b\right)=\mathrm{Op}_{\tau}\left(c_{2}\right),
	\]
	for some $c_{1},c_{2}\in W\left(\mathcal{F}L^{\infty},L_{v_{s}}^{1}\right)$.
	This means that, for fixed quantization rules $\tau,\tau_{0}$, the
	amalgam space $W\left(\mathcal{F}L^{\infty},L_{v_{s}}^{1}\right)$
	is a bimodule over the algebra $M_{1\otimes v_{s}}^{\infty,1}$ under
	the laws 
	\[
	M_{1\otimes v_{s}}^{\infty,1}\times W\left(\mathcal{F}L^{\infty},L_{v_{s}}^{1}\right)\rightarrow W\left(\mathcal{F}L^{\infty},L_{v_{s}}^{1}\right)\,\,:\,\,\left(b,a\right)\mapsto c_{1},
	\]
	\[
	W\left(\mathcal{F}L^{\infty},L_{v_{s}}^{1}\right)\times M_{1\otimes v_{s}}^{\infty,1}\rightarrow W\left(\mathcal{F}L^{\infty},L_{v_{s}}^{1}\right)\,\,:\,\,\left(a,b\right)\mapsto c_{2},
	\]
	with $c_{1}$ and $c_{2}$ as before. 
\end{remark}
In conclusion, we provide a result whose proof easily follows by \cite[Theorem 3.7]{generalizedmetaplectic} and by noticing that $\mathcal{U}_{\tau}^{-1}=\mathcal{U}_{1-\tau}$
for any $\tau\in\left(0,1\right)$.
\begin{theorem}[Wiener property]
	For any $\tau\in\left(0,1\right)$ and $a\in W\left(\mathcal{F}L^{\infty},L_{v_{s}}^{1}\right)$
	such that $\mathrm{Op}_{\tau}\left(a\right)$ is invertible on $L^{2}\left(\mathbb{R}^{d}\right)$,
	we have 
	\[
	\mathrm{Op}_{\tau}\left(a\right)^{-1}=\mathrm{Op}_{1-\tau}\left(b\right)
	\]
	for some $b\in W\left(\mathcal{F}L^{\infty},L_{v_{s}}^{1}\right)$.
\end{theorem}

\section{Acknowledgments}  The first and second authors were partially supported by  the Gruppo Nazionale per l'Analisi Matematica, la Probabilit\`a e le loro Applicazioni (GNAMPA) of the Istituto Nazionale di Alta Matematica (INdAM). The authors wish to thank the referees for their suggestions, which improved the readability of the manuscript.

\bibliographystyle{amsalpha}

\providecommand{\bysame}{\leavevmode\hbox to3em{\hrulefill}\thinspace}
\providecommand{\MR}{\relax\ifhmode\unskip\space\fi MR }
\providecommand{\MRhref}[2]{%
	\href{http://www.ams.org/mathscinet-getitem?mr=#1}{#2}
}
\providecommand{\href}[2]{#2}

\end{document}